\documentclass [12pt, reqno]{amsart}

\usepackage{amsmath, amsfonts, amssymb, setspace, color, enumerate, bbold}

      \newtheorem{theorem}{Theorem}[section]

      \newtheorem{remark}[theorem]{Remark}
       \newtheorem{corollary}[theorem]{Corollary}
      \newtheorem{lemma}[theorem]{Lemma}

      \newtheorem{convention}[theorem]{Convention}
      \def\N{{\mathbb N}}
      \def\C{{\mathbb C}}

      \def\cK{\mathcal K}
      
      \def\cM{\mathcal M}

      \def\cA{\mathcal A}
      \def\cB{\mathcal B}
      
      \def\cC{\mathcal C}
      \def\cE{\mathcal E}
      \def\cS{\mathcal S}

      \def\cK{\mathcal K}
      
      \def\tti{\mathtt i}
      
      \def\bb1{\mathbb 1}
      
      \def\fT{\mathfrak T}

\makeatletter
\@namedef{subjclassname@2020}{\textup{2020} Mathematics Subject Classification}
\makeatother
\usepackage{ulem}

\newcommand{\df}[1]{{\bf{#1}}{\index{#1}}}

\newcommand{\trace}{\operatorname{trace}}
\newcommand{\range}{\operatorname{Range}}

\usepackage{ulem}

\usepackage[dvipsnames]{xcolor}
\usepackage[pagebackref,colorlinks,linkcolor=Maroon,citecolor=MidnightBlue,urlcolor=NavyBlue,hypertexnames=true]{hyperref}

\makeindex

\title[$C^*$-extreme UEB maps]{$C^*$-extreme entanglement breaking maps on operator systems}

\author[S. Balasubramanian]{Sriram Balasubramanian${}^*$}
\address{Department of Mathematics\\
IIT Madras, Chennai - 600036, India.}
\email{bsriram@iitm.ac.in, bsriram80@yahoo.co.in}
\thanks{${}^*$ Supported by the grant MTR/2018/000113 from the Department of Science and Technology (DST), Govt. of India.}

\author[N. Hotwani]{Neha Hotwani${}^1$}
\address{Department of Mathematics\\
IIT Madras, Chennai - 600036, India.}
\email{ma18d016@smail.iitm.ac.in}
\thanks{${}^1$ Supported by the fellowship 0203/16(8)/2018-R\&D-II from the National Board for Higher Mathematics (NBHM),
 Govt. of India.}



\numberwithin{equation}{section}

\begin{document}

\begin{abstract}
Let $\cE$ denote the set of all unital entanglement breaking (UEB) linear maps 
defined on an operator system $\cS \subset M_d$ and, 
mapping into $M_n$. As it 
turns out, the set $\cE$ is not only convex in the classical 
sense but also in a quantum sense, namely it is 
$C^*$-convex. 
The main objective of this article is to describe the 
$C^*$-extreme points of this set $\cE$. By observing that every 
EB map defined on the operator system $\cS$ 
dilates to a positive map with 
commutative range and also extends to an EB map on $M_d$, we
 show that the $C^*$-extreme points of the set $\cE$ are precisely the 
UEB maps that are maximal in the sense 
of Arveson (\cite{A} and \cite{A69}) and that they are also 
exactly the linear extreme points of the set $\cE$ with commutative 
range. We also determine 
their explicit structure, thereby obtaining operator system 
generalizations of the analogous structure theorem and 
the Krein-Milman type theorem given in \cite{BDMS}. As a 
consequence,  
 we show that $C^*$-extreme (UEB) maps in $\cE$ 
extend to $C^*$-extreme UEB maps on the full algebra. 
Finally, we obtain an improved version of the main result in 
\cite{BDMS}, which contains various characterizations of 
$C^*$-extreme UEB maps between the algebras 
$M_d$ and $M_n$. 
\end{abstract}

\maketitle

%
%\begin{center}
%\df{}
%\end{center}

\section{Introduction}
\label{sec:intro}
The notions of positivity and convexity are fundamental to Mathematical analysis and 
in particular, to the theory of $C^*$-algebras. Among positive maps between $C^*$-algebras, the ones 
that are completely so, are of considerable interest. 
The study of completely positive maps was initiated by Stinespring and Arveson
 in the seminal papers \cite{S}, \cite{A69} and \cite{A72}.
 Among various 
results of significant importance in \cite{A69}, of 
particular interest to us is 
an abstract characterization of the (linear) extreme points of the convex set of completely 
positive maps between a $C^*$-algebra $\mathcal A$ and $B(H)$ for some 
Hilbert space $H$, in terms of the (minimal) Stinespring dilation. 
As important as classical convexity is, it still has some limitations in the 
non-commutative setting. Two "non-commutative" convexity 
notions that have gathered significant attention recently are matrix-convexity 
and $C^*$-convexity, the former introduced and studied by Webster
 and Winkler in \cite{W} and \cite{WW} and the latter 
 by Hopenwasser, Loebl, Moore and Paulsen in \cite{LP} and \cite{HMP}. 
Our main focus in this article is on the latter. Although these notions appear to 
be similar, they are vastly different as 
was pointed out by Farenick in \cite{F}. Farenick and Morenz 
 also obtained a complete characterization of $C^*$-extreme 
unital completely positive maps between $C^*$-algebras  
 in \cite{FM93} and \cite{FM97}. Further contributions on this and 
related topics can be found in \cite{BBK} and \cite{BK}. 

Our main objects of focus in this article are entanglement breaking (EB) maps. These maps 
have drawn considerable attention recently and are particularly sought after 
in quantum information theory. In the finite dimensional setting, an important aspect of the
 set of unital entanglement breaking (UEB) maps is that it is $C^*$-convex. This 
 warrants the study of $C^*$-extreme UEB maps. 
A complete description of such maps between matrix algebras was obtained in \cite{BDMS}.
The purpose of this article is four fold. Firstly, we obtain the explicit structure of 
 UEB maps defined on an operator system of matrices 
 that are maximal with respect to the dilation order (see \cite{A}) on the set of 
UEB maps (See Theorem \ref{thm:max-iff-stdform}). As a consequence 
we show that every UEB map on such an operator system dilates to a maximal UEB map
 (See Theorem \ref{thm:maxdil-of-UEB}). 
Secondly, we show that UEB maps on an operator system of matrices 
that are maximal with respect to the dilation order on the set of 
UEB maps are precisely the $C^*$-extreme UEB maps and that they are also exactly the linear 
extreme UEB maps with commutative range (See Theorem \ref{thm:max-iff-C*ext}).
As a consequence, we obtain an operator system generalization 
of a characterization of $C^*$-extreme UEB maps
as well as the Krein-Milman type theorem for UEB maps 
given in \cite{BDMS}. (See Corollary \ref{cor:kre-mil-for-ueb}).
It is to be noted that the partial order used for obtaining the abstract characterization of 
$C^*$-extreme UEB maps 
in \cite{BDMS} is the usual partial order (and not the dilation 
order) on the set of UEB maps, 
 i.e., for UEB maps $\Phi$ and $\Psi$, $\Phi \le \Psi$ if and only
 if $\Psi - \Phi$ is a UEB map.
Thirdly, we show that a $C^*$-extreme UEB map defined 
on an operator system of matrices extends to a $C^*$-extreme 
UEB map on the full matrix algebra (See Corollary \ref{cor:extn}). 
Finally, we obtain an improved version of the main result Theorem 5.3 in 
\cite{BDMS} (See Theorem \ref{thm:improv}). To help us prove the above 
mentioned assertions, we make use of the following key observations namely, 
an EB map defined on an operator system of matrices 
dilates to a positive map with commutative range (See Theorem \ref{thm:eb}) and also 
has an EB extension to the full matrix algebra (See Theorem \ref{thm:eb-extn}). 
We conclude the article with an example of a $C^*$-extreme UEB map on an operator 
system (See Section \ref{s:eg}).

Before we explicitly state our main observations and results, we introduce some notations 
and definitions. 
Throughout this article, $L$ and $K$ will denote separable complex Hilbert spaces and $B(L)$ will 
denote the C*-algebra of bounded linear maps defined on $L$. An operator system $\cM \subset B(L)$ is a 
self adjoint subspace containing the identity operator $I_L$.   
A linear map $\Phi: \cM \to B(K)$ is said to be {\bf unital} if $\Phi(I_L) = I_K$, {\bf positive} if $\Phi(A)$ is
 a positive operator in $B(K)$ (in this case we write $\Phi(A) \succeq 0$ or $\Phi(A) \in B(K)^+$), 
whenever $A$ is a positive operator in $\cM$.  A unital positive linear functional on $\cM$ is
 called a {\bf state}. The notation $M_k$ will denote the space of all $ k \times k$ complex
 matrices and $M_k(\cM)$, the operator system of all $k \times k$ matrices with entries
 from $\cM$. For a given linear map $\Phi: \cM \to B(K)$ and $k\in \N$, the {\bf k$^{th}$-ampliation} 
 $\Phi_k: M_k(\cM) \to M_k(B(K))$ is defined by $\Phi_k([A_{i,j}]) = [\Phi(A_{i,j})]$.
 Under the identification of $M_k(\cM)$ with $M_k \otimes \cM $, one sees that
 $\Phi_k := \tti_k \otimes \Phi$, where $\tti_k$ is the identity operator on $M_k$. 
More specifically, $\Phi_k: M_k \otimes\cM \to M_k \otimes B(K)$ is the linear map 
determined by $\Phi_k(X \otimes A)=X \otimes \Phi(A).$ The linear map 
$\Phi:\cM \rightarrow B(K)$ is said to be \textbf{completely positive (CP)} if 
 $\Phi_k$ is positive for all $k\in \N$. 

Next we define the entanglement 
breaking property of linear maps. There are a number of competing 
definitions for the notion of an entanglement breaking map $\Phi:\cM \to B(K)$ 
that agree with the usual notion in the case 
 that $\cM = B(L)$ and $K$ is a finite dimensional Hilbert space. 
In any case, they all reference the cone of (separable) matrices, 
\[
M_k^+ \otimes B(K)^+ := \left \{ \sum_{m=1}^{\ell}A_m \otimes B_m \,:\,\ell \in \N, 
 A_m \in M_k^+, B_m \in B(K)^+ \right\}.
\]

Since we are mainly interested in the case where $\cM \subset B(L)$ is an operator system and, $K$ and $L$ are finite dimensional Hilbert spaces, 
inspired by the various notions of separability introduced in \cite{CH}, 
 we will say that the linear map $\Phi:\cM\to B(K)$ is \textbf{entanglement breaking (EB)} if
 $(\tti_k \otimes \Phi)(X) \in \overline{M_k^+ \otimes B(K)^+}$, for all 
$k\in \N$ and $X \in (M_k \otimes \cM)^+$, where the closure is with respect to 
the norm topology on $B(\C^k \otimes K)$.   Evidently, when $K$ is finite dimensional,
 $M_k^+\otimes B(K)^+$ is already norm closed.  The abbreviations {\bf UCP} and {\bf UEB} will 
be used for "unital completely positive" and "unital entanglement breaking" respectively.
The collection of all UEB maps mapping $\cM$ to $B(K)$ will be denoted by UEB($\cM, B(K))$. Finally, since we work mainly in the finite dimensional setting, we make the following notational 
conventions. 
\begin{convention}
\label{assump:main}
 Throughout $E = \C^d$ and $H = \C^n$ and 
$\cS \subset B(E)$ is an operator system. 
\end{convention}

%Before we state our main results, a few more definitions are in order. 
Given a linear map $\Phi:\cS \rightarrow B(H)$, the \df{dual functional} $s_\Phi: B(H) \otimes \cS \to \C$ 
associated to it, is the linear mapping determined by 
\[s_\Phi(X \otimes A)= \trace(\Phi(A)X^t),
\]
where $t$ is the transpose operator in $B(H)$ induced by a fixed orthonormal basis of $H$. 
%Note that 
%the above definition is independent of the choice of the transpose operator
 It is easily seen that this
 definition of $s_{\Phi}$ is independent of the choice of the orthonormal basis and hence of the transpose 
operator induced by it and that the correspondence $\Phi \mapsto s_{\Phi}$ is bijective.  Please 
see  \cite{St}, \cite{B} for more details.

The \textbf{Choi matrix} $C_\Phi$ associated with the linear map $\Phi:B(E) \to B(H)$ is defined as 
$[\Phi(e_ie_j^*)]_{i,j=1}^d= \sum_{i,j=1}^d e_ie_j^* \otimes \Phi(e_ie_j^*)\in B(E) \otimes B(H)$,
 where $\{e_1, \dots, e_d\}$ is the standard orthonormal basis of $E$. 
One of the many significant 
applications of the Choi-matrix is the following well-known characterization of CP maps due 
to Choi.

\begin{theorem} {(\cite[Theorem 1]{C}, \cite[Theorem 4.1.8]{St})}
\label{thm:cpmaps-choikraus}
Let $\Phi: B(E) \rightarrow B(H)$ be a linear map. The  
following statements are equivalent. 
\begin{itemize}
\item[(i)] $\Phi$ is CP.
%\label{Choi-Krauss form}
\item[(ii)] $\Phi(X) = \sum_{k=1}^\ell V_k^*XV_k$,  for some linear maps $V_k:H \to E$.
\item[(iii)] $C_{\Phi} \in (B(E) \otimes B(H))^+.$
\end{itemize}
\end{theorem}
%}
The formula  for $\Phi$ given in statement $(ii)$ above is called a 
{\bf Choi-Kraus decomposition} of $\Phi$ and the matrices coefficients 
$V_k$ are called {\bf Choi-Kraus operators/coefficients}. The minimum 
 number of Choi-Kraus operators required to represent $\Phi$ in the form of a Choi-Kraus 
decomposition is known as the \textbf{Choi-rank of $\Phi$}.
% It is known that such Choi-Kraus operators must necessarily be linearly independent.  
 The \df{Schmidt rank of a vector 
$\xi \in 
E \otimes H$}, denoted by $SR(\xi)$, is the smallest natural number $k$ such that 
$\xi = \sum_{i=1}^k x_i \otimes y_i 
\in E \otimes H$. Given a completely positive map $\Phi:B(E) \rightarrow B(H)$, 
by Theorem \ref{thm:cpmaps-choikraus} the Choi matrix $C_\Phi \in 
(B(E) \otimes B(H))^+$ and hence has a spectral decomposition of the form 
 $C_\Phi = \sum_{i=1}^m \xi_i\xi_i^*$, where $\xi_i \in E \otimes H$. Let 
$\Pi_\Phi$ denote the set of all spectral decompositions of $C_\Phi$.
The \df{Schmidt number of the Choi matrix $C_{\Phi} \in (B(E) \otimes B(H))^+$} is 
 denoted by $SN(C_\Phi)$ and  is 
defined as 
\[
SN(C_\Phi):= \min_{\Gamma \in \Pi_\Phi} \left \{ \max_{1 \le j \le m} \left \{ SR(\xi_j)\,  \middle \vert \, \Gamma = \sum_{i=1}^m \xi_i\xi_i^* \in \Pi_\Phi \right \} \right\}.
\]
See \cite{TH} for more details.
%}

In this article, our main focus will be on EB maps. 
The following are a few well-known characterizations of EB maps 
that will be used in the sequel. 

\begin{theorem}(\cite[Theorem 4]{HSR}, \cite{TH})
\label{thm:ebmaps}
Let $\Phi:B(E) \rightarrow B(H)$ be a linear map. 
%where $E$ and $H$ are finite dimensional Hilbert spaces. 
The following statements are equivalent. 
\begin{enumerate}[(i)]
\item \label{i:ebmaps1} $\Phi$ is EB.
\item \label{i:ebmaps2} $\Phi(X) = \sum_{j=1}^m \phi_j(X)R_j$, where the $\phi_j$'s are states defined on $B(E)$ and the $R_j$'s are positive operators in $B(H)$.
%\label{Choi-Krauss form}
\item \label{i:ebmaps3} $\Phi(X) = \sum_{k=1}^\ell V_k^*XV_k$, where each $V_k:H \to E$ is a linear map of rank one.
\item $C_{\Phi} \in B(E)^+ \otimes B(H)^+$. (i.e., $C_{\Phi}$ is separable.)
\item $SN(C_{\Phi}) = 1$.
\end{enumerate}
\end{theorem}

The formula for $\Phi$ given in statement \eqref{i:ebmaps2} of the above theorem is called a 
{\bf Holevo form of $\Phi$}. The \textbf{EB-rank of $\Phi$} is defined as the
minimum number of Choi-Kraus operators of rank one,
 required to represent $\Phi$ as in statement \eqref{i:ebmaps3} of Theorem \ref{thm:ebmaps}.
Please see \cite{PPPR} for more details.

"Extremal" UEB maps are of  utmost importance to us.  Let $\cS \subset B(E)$ be an operator system. 
Here we mainly consider two notions of extreme points of UEB($\cS, B(H))$. 
The UEB map $\Phi: \cS \to B(H)$ is said to be a {\bf linear extreme point} of UEB($\cS, B(H))$ if 
\[
\Phi= \sum_{i=1}^k t_i \Phi_i
\]
 for some UEB maps $\Phi_i: \cS \to B(H)$ and $t_i \in (0,1)$ satisfying $\sum_{i=1}^k t_i=1$,
 then $\Phi_i=\Phi$ for all $i$. Linear extreme UCP maps betweeen $C^*$-algebras were characterized by Arveson in \cite{A69}.
 
The notion of $C^*$-convexity and $C^*$-extreme points were first introduced and studied in \cite{LP} and 
\cite{HMP} respectively. $C^*$-extreme UCP maps were extensively studied in \cite{FM93}, \cite{FM97}, \cite{FZ} and \cite{Z}. 
In \cite{BDMS}, the authors consider $C^*$-extreme UEB maps between matrix algebras. 
Along the same lines, we define $C^*$-extreme points in UEB($\cS, B(H))$.

Let $\Phi_1, \Phi_2,\dots,\Phi_k \in \text{UEB}(\cS, B(H))$. A UEB map $\Phi:\cS \rightarrow B(H)$ is 
said to be a \df{$C^*$-convex combination} of the UEB maps $\Phi_i$, if there exists 
$T_1,\dots,T_k \in B(H)$ such that $\sum_{i=1}^k T_i^*T_i = I_H$ and 
$\Phi(X)= \sum_{i=1}^k T_i^* \Phi_i(X) T_i$ for every $X \in \cS$. The $T_i$'s will be 
referred to as the \df{coefficients} of this $C^*$-convex combination. If the 
coefficients, i.e., the $T_i$'s are positive, then this $C^*$-convex combination will be called \df{a 
positive $C^*$-convex combination}.

%If, in addition, $T_i$'s are invertible, 
%then $\Phi$ is said to be a \df{proper $C^*$-convex combination} of the UEB maps 
%$\Phi_i$. 

A UEB map $\Phi: \cS \to B(H)$ is a \df{proper $C^*$-convex combination} of the UEB maps $\Phi_i$, 
$1 \le i \le k$, if there exists invertible operators $T_i \in B(H)$ such that 
$\sum_{i=1}^k T_i^*T_i = I_H$ and 
\begin{equation}
\label{eq:C*-convex combi}
\Phi(X)= \sum_{i=1}^k T_i^* \Phi_i(X) T_i,
\end{equation}
for all $X\in \cS.$ 
This $C^*$-convex combination is 
\df{trivial} if each $\Phi_i$ is unitarily equivalent to $\Phi;$ that is, 
there exist unitary operators $U_i$ such that  $\Phi_i(X)=U_i^*\Phi(X) U_i.$ 
The UEB map $\Phi: \cS \to B(H)$ is said to be a {\bf $C^*$-extreme point} of 
UEB($\cS, B(H))$ if, every representation of  $\Phi$ 
as a proper $C^*$-convex combination is trivial. 

% Let $\cK \subset \text{UEB}(\cS, B(H))$.
 The {\bf $C^*$-convex hull} of the set $\cK \subset \text{UEB}(\cS, B(H))$ is 
the set of all $C^*$-convex combinations of 
elements of $\cK$. The set $\cK$ is said to be {\bf $C^*$-convex} if it 
equals its $C^*$-convex hull.

%The following is an important 
%observation about the set of UEB maps.

%\begin{theorem}
%\label{UEB maps-are-C*-convex}
%$\text{UEB}(\cS, B(H))$ is C*-convex.
%\end{theorem}
%
%\begin{proof}
%Let $1 \le i \le k$, $\Phi_i \in$ UEB($\cS, B(H)$) and $T_i \in B(H)$ 
%be such that $\sum_{i=1}^k T_i^*T_i = I_H$. Suppose that $\Phi$ is as in 
%equation \eqref{eq:C*-convex combi}. It suffices to show that 
%$\Phi \in$ UEB($\cS, B(H)$). By Corollary 
%\ref{cor:ebmaps-on-S}, for each $i \in \{1, 2, \dots, k\}$, there exists 
% $\ell_i \in \N$ and rank one matrices $V_{i,1}, \dots, V_{i, \ell_i}$ such that 
%\[
%\Phi_i(X) = \sum_{j=1}^{\ell_i} V_{i,j}^*XV_{i,j}.
%\]
%Since each $\Phi_i$ is unital,  
%\[
%\sum_{j=1}^{\ell_i} V_{i, j}^*XV_{i, j} = I_H,
%\]
%Thus,  
%\[
%\Phi = \sum_{i=1}^k \sum_{j=1}^{\ell_i} (T_i^* V_{i, j}^*) X (V_{i, j}T_i).
%\]
%Observe that $V_{i, j}T_i$ is of rank one, for every $1 \le i \le k$ and $1 \le j \le \ell_i$. 
%Another application of Corollary \ref{cor:ebmaps-on-S} completes the proof. 
%\end{proof}

\begin{remark}
The set $\text{UEB}(\cS, B(H))$ is C*-convex. This can be seen as follows. 
For $1 \le i \le k$, let $\Phi_i \in$$\text{UEB}(\cS, B(H))$, $T_i \in B(H)$ 
be such that $\sum_{i=1}^k T_i^*T_i = I_H$ and $\Phi$ be as in 
equation \eqref{eq:C*-convex combi}. Indeed $\Phi$ is unital.  
 Define $\Gamma_i:\cS \rightarrow B(H)$ by $\Gamma_i(X) = T_i^*\Phi_i(X)T_i$.  
Since each $\Phi_i$ is a UEB map, 
for each $m \in \N$ and $Z \in (M_m \otimes \cS)^+$, it is the case that 
$(\tti_m \otimes \Phi_i)(Z) \in M_m^+ \otimes B(H)^+.$ 
Let $(\tti_m \otimes \Phi_i)(Z) := \sum_{j=1}^{\ell_i} A_{i,j} \otimes B_{i,j} 
\in M_m^+ \otimes B(H)^+$. 
It follows that 
$(\tti_m \otimes \Gamma_i)(Z) 
%= (\tti_m \otimes T_i^*) (I_m \otimes \Phi)(Z) (\tti_m \otimes T_i) 
= 
\sum_{j=1}^{\ell_i} A_{i,j} \otimes T_i^* B_{i,j} T_i  \in M_m^+ \otimes B(H)^+.
$
Thus $(\tti_m \otimes \Phi)(Z) = \sum_{i=1}^k (\tti_m \otimes \Gamma_i)(Z) \in 
M_m^+ \otimes B(H)^+$ and $\Phi \in$ $\text{UEB}(\cS, B(H))$. 
\end{remark}

\begin{remark} 
\label{rem:cstar-ext-is-linearext}
Given our finite dimensionality assumptions, it is a standard observation that
 $C^*$-extreme UEB maps from $\cS \subset B(E)$ mapping into $B(H)$ are 
also linear extreme UEB maps. Please see \cite[Theorem 2.2.2]{Z} for a 
proof of this fact for UCP maps. However, not every linear extreme UEB map 
is $C^*$-extreme, as an example in \cite{HSR} shows. See also 
example 5.7 in \cite{BDMS}.
\end{remark}

%The following Lemma contains an equivalent definition of a $C^*$-extreme UEB map and 
%is essentially the operator system version of \cite[Proposition 3.2]{BDMS}.
%
%\begin{lemma}
%\label{lem:alt-defn}
%Let $\cS \subset B(E)$ be an operator system and $\Phi:\cS \rightarrow B(H)$ be a UEB map. The following statements are equivalent.
%\begin{itemize}
%\item [(i)] $\Phi$ is a $C^*$-extreme point of UEB($\cS, B(H$)).
%\item[(ii)] If $\Phi = \sum_{i=1}^2 T_i^* \Phi_i T_i$ for some invertible operators $T_1, T_2 \in B(H)$ satisfying 
%$T_1^*T_1 + T_2^*T_2 = I_H$, then there exist unitaries $U_1, U_2 \in B(H)$ such that 
%$\Phi_i(X)=U_i^*\Phi(X) U_i$, $i= 1, 2$.
%\end{itemize}
%\end{lemma}

With the above given definitions and remarks, we proceed to 
state our main observations and results. 
%}
\subsection{EB maps - Extension and Structure} A well-known extension theorem for CP maps
due to  Arveson says that every CP map defined on an operator system in a $C^*$-algebra 
mapping into $B(K)$ for some Hilbert space $K$, has a CP extension (See \cite[Theorem 7.5]{P}). 
Our first main observation in this article is the following analogous extension theorem for EB maps. 

\begin{theorem}
\label{thm:eb-extn}
Let $\cS \subset B(E)$ be an operator system and $\phi: \cS \to B(H)$ be an EB map. There exists an EB map $\Phi: B(E) \to B(H)$ such that $\Phi\big|_\cS= \phi$.
\end{theorem}

A proof of Theorem \ref{thm:eb-extn} is given in Section \ref{S:extn}. As an immediate consequence of the above theorem, one obtains the following operator system version of Theorem \ref{thm:ebmaps}.

\begin{corollary}
\label{cor:ebmaps-on-S}
The equivalence of statements \eqref{i:ebmaps1}, \eqref{i:ebmaps2} and \eqref{i:ebmaps3} of Theorem~\ref{thm:ebmaps} holds even for UEB maps defined on operator systems. More precisely, if $\cS \subset B(E)$ is an operator system and 
$\Phi:\cS \rightarrow B(H)$ is  a linear map, then the following statements are equivalent.
\begin{enumerate}
\item[(i)]$\Phi$ is EB.
\item[(ii)] $\Phi$ can be written in the Holevo form, i.e., $\Phi(X) = \sum_{j=1}^m \phi_j(X)R_j$,
 where the $\phi_j$'s are states defined on the operator system $\cS$ and the $R_j$'s are positive operators in $B(H)$.
\label{Choi-Krauss form}
\item[(iii)] $\Phi$ has a Choi-Kraus decomposition with rank-one Choi-Kraus coefficients, i.e., there exist linear maps $V_k:H \to E$ of rank one such that $\Phi(X) = \sum_{k=1}^\ell V_k^*XV_k$, 
for all $X \in \cS$.
\end{enumerate}
\end{corollary}

It is a well-known result due to Stinespring (\cite[Theorem 4.1]{P}) that a UCP map defined on a C*-algebra $\cA$
mapping into $B(K)$ for some Hilbert space $K$, dilates to a representation. In this article we point out a similar structure theorem for EB maps (for our finite dimensional setting). The following characterization of an EB map in terms of a dilation, is our second main observation 
in this article. 

\begin{theorem}
\label{thm:eb}
Let $\cS \subset B(E)$ be an operator system and $\Phi:\cS \rightarrow B(H)$ a linear map. The following statements are equivalent.
\begin{itemize}
\item[(i)] $\Phi$ is a (unital) EB map.
\item[(ii)] $\Phi(X) = V^* \Gamma(X) V$ for all $X\in \cS$, where $V:H \rightarrow K$ is an isometry for some  finite dimensional Hilbert space $K$ 
 and $\Gamma:\cS \rightarrow B(K)$ is a (unital) positive map with commutative range.
\item[(iii)] $\Phi$  is the compression of a (unital) EB map with commutative range contained in $B(K)$ for some finite dimensional Hilbert space $K$. 
\end{itemize}
\end{theorem}

A proof of Theorem \ref{thm:eb} is given in Section \ref{S:Dila}. It relies mainly on the 
 observation that an entanglement breaking map between matrix algebras factors via the
 commutative C*-algebra $\ell^{\infty}_k$ for some $k$. See \cite{KMP} and \cite{JKPP}
 for more details. We anticipate that the above structure theorem for "EB" maps 
being compressions of "positive" maps with commutative range should hold for a much more 
general setting than is considered here (for instance, for strongly entanglement breaking 
maps in the infinite dimensional setting (See \cite{LD})). 

\begin{remark} 
Theorem \ref{thm:eb} can also be deduced from the well-known fact that a 
 Positive Operator Valued Measure (POVM) dilates to a Projection Valued 
Measure (PVM) (see \cite{Y}), which in turn can be deduced from Naimark's dilation Theorem 
(\cite[Theorem 4.6]{P}). 
\end{remark}

\subsection{Maximal UEB maps}

In \cite{A} and \cite{A72}, Arveson defined the notion of a maximal UCP dilation of a given UCP map on an operator system. He also showed that maximal UCP dilations always exist and that such maps are precisely the ones that satisfy a unique extension property.  Following \cite{A} and \cite{A72}, here we consider the UEB analog of maximal dilations (for our finite dimensional setting.)  
Let $\cS \subset B(E)$ be an operator system and let $\Phi: \cS \to B(H)$ be a UEB map. A linear map $\Psi: \cS \to B(K)$ is said to be a \textbf{UEB dilation} of $\Phi$ if $\Psi$ is a UEB map, $K$ is a separable Hilbert space, and there exists an isometry $V: H \to K$ such that 
 \begin{equation}
\label{eq:eb-eq}
\Phi(X)= V^*\Psi(X)V, 
\end{equation}
for all $X\in \cS$. In this case we write $\Phi \le \Psi$. Note that if $H$ is identified with $VH$, then "$\le$" is a partial order on 
the set of all UEB 
maps defined on $\cS$. The  UEB dilation $\Psi$ of $\Phi$ is said to be \df{trivial} if 
 \begin{equation}
 \label{eq:trivial-dil-eq}
 \Psi(X)V(x)=V(\Phi(X)(x)),
\end{equation}  for all $X\in \cS$ and $x\in H$. The UEB map $\Phi$ is  said to be {\bf maximal} if every UEB dilation of $\Phi$ is trivial.

\begin{remark}
 \label{rem:alt-defn}

The following observations are immediate from the above definition of maximality for UEB maps.
\begin{itemize}
\item [(i)] Since $VV^*$ is the projection onto $\range(V)$, it follows that the UEB dilation $\Psi$ is trivial if and only if $VH$ is an invariant subspace for $\Psi(X)$ for all $X \in \cS$. Also, since $\cS$ is an operator system and $\Psi(X^*)=\Psi(X)^*$ for all $X\in \cS,$ it follows that the 
UEB dilation $\Psi$ of $\Phi$ is trivial if and only if $VH$ is a reducing subspace for $\Psi(X)$ for every $X \in \cS$. 
\item[(ii)] Since $H$ is finite dimensional and $V:H \rightarrow K$ is an isometry, there is no loss of generality 
in assuming $K$ to be finite dimensional. (See Lemma \ref{lem:fin-max}) 
\end{itemize}
\end{remark}
The following characterization of maximal UEB maps is one of our main results in this article. 

\begin{theorem}
\label{thm:max-iff-stdform}
Let $\cS \subset B(E)$ be an operator system and $\Phi: \cS \to B(H)$ be a UEB map. The following statements are equivalent.
\begin{itemize}
\item[$(i)$] $\Phi$ is maximal.
\item[$(ii)$] $\Phi$ has the form 
\begin{equation}
\label{eq:stdform}
 \Phi(X)= \sum_{i=1}^k \phi_i(X)P_i ,
\end{equation}
for all $X\in \cS$, where the $\phi_i$'s are distinct linear extremal states on $\cS$, $k\leq n$ and the $P_i$'s are mutually orthogonal projections in $B(H)$ such that $\sum_{i=1}^k P_i=I_H$.
\end{itemize} 
\end{theorem}

A proof of Theorem \ref{thm:max-iff-stdform} is given in Section \ref{S:MaxUEB}. 
Using Theorems \ref{thm:eb} and \ref{thm:max-iff-stdform}, we also
observe the following.

\begin{theorem}
\label{thm:maxdil-of-UEB}
Let $\cS \subset B(E)$ be an operator system. Every UEB map $\Phi:\cS \to B(H)$ dilates
 to a maximal UEB map $\Psi:\cS \to B(K)$ with $dim(K) < \infty$.
\end{theorem}
A proof of Theorem \ref{thm:maxdil-of-UEB} can be found in Section \ref{S:MaxUEB}.

\subsection{C*-extreme UEB maps} In \cite{BDMS}, the authors obtain various
 characterizations of $C^*$-extreme UEB maps between matrix algebras. 
A primary objective of ours is to characterize $C^*$-extreme UEB maps defined
 on operator systems of matrices. One of our main results in this article 
along these lines is the following. 

\begin{theorem}
\label{thm:max-iff-C*ext}
Let $\cS \subset B(E)$ be an operator system
 and $\Phi: \cS \to B(H)$ a UEB map. Then the following statements are equivalent.
\begin{itemize}
\item[$(i)$] $\Phi$ is a maximal.
\item[$(ii)$] $\Phi$ is a $C^*$-extreme point of UEB$(\cS, B(H))$.
\item[$(iii)$] $\Phi$ is a linear extreme point of UEB$(\cS,B(H))$ with commutative range.
\end{itemize} 
\end{theorem}
%{ 
A proof of Theorem \ref{thm:max-iff-C*ext} is given in Section \ref{S:C*-ext}. 
Combined with Theorem~\ref{thm:max-iff-stdform}, the equivalence $(i) \iff (ii)$ of 
Theorem~\ref{thm:max-iff-C*ext} is an operator 
system generalization of the equivalence $(i) \iff (v)$ of 
\cite[Theorem~5.3]{BDMS}, which says that $\Phi \in$ UEB($B(E), B(H)$) is $C^*$-extreme
 if and only if it has the form given in equation \eqref{eq:stdform}. \\  %}

After establishing a characterization as given above, 
it is only natural to ask whether $C^*$-extreme UEB maps defined on operator systems in 
$B(E)$ extend to $C^*$-extreme UEB maps on the whole algebra. This is still only a partially 
answered question for $C^*$-extreme UCP maps. Please see \cite{Z} for more details and 
some results on this problem. Here, as an application of Theorems  
\ref{thm:max-iff-stdform} and \ref{thm:max-iff-C*ext}, we obtain the following extension 
result.

%we show that,  every $C^*$-extreme 
%UEB map $\Phi:\cS \subset B(E) \rightarrow B(H)$ extends to a $C^*$-extreme UEB map 
%on $B(E)$.

\begin{corollary}
\label{cor:extn}
Let $\cS \subset B(E)$ be an operator system and $\Phi : \cS \to B(H)$  be a $C^*$-extreme UEB 
map. There exists a $C^*$-extreme UEB map $\Psi :B(E) \to B(H)$ such that $\Psi|_{\cS} = \Phi$.
\end{corollary}
A proof of Corollary \ref{cor:extn} is given in Section \ref{S:C*-ext}.\\

In \cite{BDMS}, a Krein-Milman type theorem was established for the 
compact convex set UEB$(B(E),B(H))$. To be precise, it was shown that the
 set UEB$(B(E),B(H))$ equals the $C^*$-convex hull of its $C^*$-extreme points. 
As another application of our main results, we obtain the following operator system analog of this 
result. 
%we show that the same conclusion also 
%holds for UEB$(\cS, B(H))$, where $\cS \subset B(E)$ is an operator system. 

\begin{corollary}
\label{cor:kre-mil-for-ueb}
\text{UEB}$(\cS,B(H))$ equals the $C^*$-convex hull of its $C^*$-extreme points.
\end{corollary}

A proof of Corollary \ref{cor:kre-mil-for-ueb} is given in Section \ref{S:C*-ext}.

\section{Extensions of EB maps}
\label{S:extn}

%\begin{definition} 
%\end{definition}

%\begin{definition}
%Let $\cC$ be a mapping cone in $\mathcal{P}(n)$, and let $\cS\subset M_d$ be an operator system. Then $P(\cS, \cC)$ is defined by
%\[
%P(\cS, \cC)=\{X\in SA(\cMn \otimes \cS): (\alpha \otimes \tti)(X) \succeq 0 \quad \textrm{ for all $\alpha \in \cC$}\},
%\]
%where $\tti$ is the identity map on $\cS$.
%\end{definition}

%\begin{definition}

%\end{definition}

In this section we give a proof of Theorem \ref{thm:eb-extn}. 
%The idea of the proof is to first show that every EB map defined on an operator system $\cS\subset B(E)$ has an approximate EB extension on the entire $B(E)$ and then use this to prove it has an actual extension. In this direction, 
We first include some necessary definitions adapted to our finite dimensional setting.  Please see \cite[Chapter 5]{St} for more details. 

 %(\cite[Definition 5.1.1]{St}) 
Let $L$ and $M$ be finite dimensional Hilbert spaces. Let $\mathcal P(M)$ denote the cone of all positive linear maps from $B(M)$ to $B(M)$. 
A closed convex cone $\mathcal{C} \subset \mathcal P(M)$ is said to be a \df{mapping cone} if for each nonzero $A \in B(M)^+$, there exists a $\varphi\in \cC$ such that $\varphi(A) \neq 0$ and
\[
 \phi \circ \sigma \circ \psi \in \mathcal{C},
 \]
  for all $\sigma \in \mathcal{C}$ and CP maps $\phi,\psi: B(M) \to B(M)$. 
%{
The mapping cone $\cC$ is said to be \textbf{symmetric} if $\phi \in \cC$ implies both $\phi^*$ 
and $t \circ \phi \circ t$ are in $\cC$, 
%}
where $t$ is the transpose operator in $B(M)$ induced by a fixed
 orthonormal basis of $M$ and $\phi^*:B(M) \rightarrow B(M)$ is {\bf the adjoint of $\phi$} 
with respect to the Hilbert-Schmidt inner product on $B(M)$, i.e., $\phi^*$ is determined by 
\[
\langle \phi(A), B\rangle= \langle A, \phi^*(B)\rangle,
\]
for all $A,B \in B(M)$, where $\langle C, D \rangle := \trace(D^*C)$. 
\begin{remark}
Typical examples of mapping cones that are symmetric 
are UCP$(B(M))$ and EB$(B(M))$.
\end{remark}

Suppose that $\cS\subset B(L)$ is an operator system and $\mathcal{C} \subset \mathcal P(M)$ is a mapping cone.
 A linear map $\phi: \cS \to B(M)$ is said to be \df{$\mathcal{C}$-positive} if the corresponding dual functional 
$s_\phi:B(M) \otimes \cS \to \C$ takes positive values on the cone 
\[
P(\cS, \mathcal{C}) := \{ Y \in B(M) \otimes \cS : Y = Y^*, (\alpha \otimes\tti)(Y) \succeq 0 \text{ for all $\alpha$ in $\cC$}\},
\] where $\tti$ is the identity map on $\cS$. 
%Recall Assumption \ref{assump:main}. 
\begin{lemma}
\label{lem:mappingcone}
 Let $\cS \subset B(E)$ be an operator system and $\phi: \cS \to B(H)$ be an EB map.
 If $\cC$ denotes the mapping cone EB$(B(H))$, then $\phi$ is $\cC$-positive. 
\end{lemma} 

\begin{proof}
%{ The proof follows directly from  \cite[Theorem 5.1.13] {St}.}
%It suffices to show that the dual functional
% $s_\phi: B(H) \otimes \cS  \to \C$
%  is positive on the cone $P(\cS, \cC)$. 
Since $\phi$ is an EB map, it follows 
from {\cite[Proposition 1(ii)]{St09}} that the dual functional 
  $s_\phi$ takes positive values on the cone $K$, where 
  \[
  K:=\{ X\in (B(H) \otimes \cS) : (\omega \otimes \tti)(X)\succeq 0 \quad  \textrm{for all  states $\omega$ on $B(H)$} \}.
  \]
Thus it suffices to show that $P(\cS, \cC)\subseteq K$.
To prove this statement, let $\omega$ be a state defined on $B(H)$. Suppose 
also that $Y= \sum_i A_i \otimes B_i \in P(\cS, \cC)$ and $Z \in B(H)^+$ is a fixed matrix 
of rank one. Define the linear map $\gamma:B(H) \rightarrow B(H)$ by 
$\gamma(X) = \omega(X)Z$. Observe that $\gamma$ has the Holevo form (See Theorem 
\ref{thm:ebmaps}) and hence $\gamma \in \cC$.  Since $(\alpha \otimes \tti)(Y) \succeq 0$ for 
all $\alpha \in \cC$, it follows that 
\[
(\gamma \otimes \tti)(Y) = 
\sum_i \omega(A_i)Z\otimes B_i = Z\otimes \left[ \sum_i\omega(A_i)B_i\right]  \succeq 0 
\] 
and therefore 
%for each $Z \in B(H)^+$ of rank one. This implies that  
 $\sum_i \omega(A_i)B_i = (\omega \otimes \tti)(Y) \succeq 0 $.
Thus $Y\in K$ and the proof is complete.
\end{proof}

%The following lemma is a key ingredient that will be utilized to prove our main extension theorem.
%Next we show that every EB map on an operator system has an approximate EB extension.

\begin{lemma}
\label{lem:approx-eb-extn}
Let $\cS \subset B(E)$ be an operator system and $\phi: \cS \to B(H)$ be an EB map. Given $\epsilon >0$, 
there exists an EB map $\Phi: B(E) \to B(H)$ such that $\|\Phi(A) - \phi(A)\| < \epsilon \|A\|$ for all $A\in \cS$.
\end{lemma}
\begin{proof}
Let $\mathcal C =$ EB$(B(H))$. 
Since $\phi: \cS \to B(H)$ is an EB map, it 
follows from Lemma \ref{lem:mappingcone} that $\phi$ is $\cC$-positive.
By observing that $\cC$ is a symmetric mapping cone, it follows from \cite[Theorem 5.1.13]{St} that
 there exists a sequence $\phi_j:= \sum_{i=1}^{r_j} \alpha_i^{(j)} \circ \psi_i^{(j)} \in \cC$ 
with $\alpha_i^{(j)}\in \cC$ and CP maps 
$\psi_i^{(j)}: \cS \to B(H)$ such that
%{ (NOT SURE EXACTLY HOW TO WRITE THIS. IS THIS OKAY?)
 $\phi_j \to \phi$ in norm, in $B(H)$. 
 Note that norm topology  on $B(H)$ coincides with the BW-topology due to 
the finite dimensionality of $B(H)$. %}
%Since $\psi_i^{(j)}: \cS \to B(H)$ is a cp map,
By the Arveson extension theorem \cite[Theorem 7.5]{P},
there exist CP maps $\Psi_i^{(j)}: B(E) \to B(H)$ 
such that $\Psi_i^{(j)}\big|_\cS= \psi_i^{(j)}$.
Define $\Phi_j:= \sum_{i=1}^{r_j} \alpha_i^{(j)} \circ \Psi_i^{(j)}$.
It follows that $\Phi_j:B(E) \to B(H)$ is an EB map for all $j$.
Let $A\in \cS$ be arbitrary. Since $\phi_j \to \phi$, it follows that there exists an $j_0 \in \N$ 
such that  $\|\phi_j(A) - \phi(A)\| < \epsilon \|A\|$
for all $ j \geq j_0$. Thus for all $j\geq j_0$, one obtains 
\begin{align*}
\|\Phi_j(A)-\phi(A)\|= & \left \|\sum_{i=1}^{r_j} \left(\alpha_i^{(j)} \circ \Psi_i^{(j)}\right)(A)- \phi(A) \right\|\\
 = &  \left \|\sum_{i=1}^{r_j} \left(\alpha_i^{(j)} \circ \psi_i^{(j)} \right)(A)- \phi(A) \right \| \\
 = &  \|\phi_j(A)- \phi(A)\|\\
 < & \epsilon \|A\|.
\end{align*}
Defining $\Phi := \Phi_{j_0}$ completes the proof. 
\end{proof}

%{
\begin{remark}
Recall that $E=\C^d$. Due to the convexity of the mapping cone $\cC$ 
and the finite dimensionality of $B(E)$, Caratheodory's theorem \cite[Theorem 16.1.8]{DD} 
implies that the map $\phi_j$ in the above proof 
can be written as $\phi_j = \sum_{i=1}^{r} \alpha_i^{(j)} \circ \psi_i^{(j)}$, 
where $r = (2d)^2 +1$ (is independent of $j$). 
\end{remark}
%}

%Now we are ready to prove the EB extension theorem.

%By the above approximate EB extension result we can show that every EB map defined on an 
%operator system has an EB extension on the entire matrix algebra.

\begin{proof}[Proof of Theorem~\ref{thm:eb-extn}]
Let $j\in \N$ be arbitrary. By Lemma \ref{lem:approx-eb-extn}, there exists an EB map $\Phi_j: B(E) \to B(H)$ 
such that $\|\Phi_j(A)-\phi(A)\|<\frac{1}{j}\|A\|$ for all $A\in \cS$. Note that $\|\Phi_j\|=\|\Phi_j(I_E)\|$ and
\begin{align*}
\|\Phi_j(I_E)\|&\leq \|\Phi_j(I_E)-\phi(I_E)\|+\|\phi(I_E)\|\\
& < \frac{1}{j} +\|\phi(I_E)\|.
\end{align*}
Thus the sequence $\{\Phi_j\}_{j\in \N}$ is bounded. 
Consider the set 
\[
\mathcal K := \{\Gamma : B(E) \rightarrow B(H)\,:\, \Gamma \text{ is EB and } \|\Gamma\| \le 1 + \|\phi(I_E)\|\}.
\] 
Observe that $\mathcal K$ is compact and $\{\Phi_j\}_{j\in \N}$ is a sequence in $\mathcal K$. By the compactness of 
$\mathcal K$, there exists a subsequence $\{\Phi_{j_k}\}_{k \in \N}$ of $\{\Phi_j\}_{j\in \N}$  and 
 an EB map $\Phi \in \mathcal K$ such that $\Phi_{j_k} \rightarrow \Phi$ as 
$k \rightarrow \infty$.  The proof is complete by observing that $\Phi\big|_\cS=\phi$. 

%By the finite dimensionality of $E$ and $H$, it follows that the sequence $\{\Phi_j\}_{j\in \N}$  has a convergent subsequence $\{\Phi_{j_k}\}_{k \in \N}$. i.e. there exists a linear map $\Phi:B(E) \to B(H)$ such that $\Phi_{j_k} \rightarrow \Phi$. Since the cone of all EB maps from $B(E)$ to $B(H)$ is a closed set, it follows that $\Phi$ is an EB map. Since $\Phi\big|_\cS=\phi$, the proof is complete.
\end{proof}

\begin{remark}
It is to be noted that Theorem \ref{thm:eb-extn} does not necessarily 
follow from the extension theorem \cite[Theorem 5.2.3]{St}, because the proof 
of \cite[Theorem 5.2.3]{St} works only if the underlying operator system is a 
real operator system (i.e., a real subspace consisting of self-adjoint 
elements and the identity) as opposed to an arbitrary operator system, like 
we are considering here, particularly in Theorem \ref{thm:eb-extn}.
This gap in the proof of \cite[Theorem 5.2.3]{St} was pointed out in \cite{St18}.
 In fact, it was also remarked in \cite{St18} that the gap in the proof 
was due to the fact that Krein's Extension Theorem (\cite[Theorem A.3.1]{St}) 
which is formulated only for real spaces, was incorrectly applied to complex 
spaces. 
\end{remark}

\section{Dilations of UEB maps}
\label{S:Dila}
In this section, we prove Theorem \ref{thm:eb}. The 
proof relies on a crucial observation from \cite[Lemma 3.1]{RJP}. 
We begin with the following remark.

\begin{remark}
\label{rem:pos-to-eb}
Let $\Psi:B(E) \rightarrow B(H)$ be a (not necessarily unital) positive map with commutative range. 
The argument  in the proof of \cite[Lemma 3.1]{RJP} can easily be adapted 
 to conclude that $\Psi$ is EB. 
\end{remark}

\begin{lemma}
\label{lem:pos-to-eb}
Let $K$ be a finite dimensional Hilbert space and $\cS \subset B(E)$ be an operator system. 
If $\Gamma:\cS \rightarrow B(K)$ is a positive map with commutative range, 
then $\Gamma$ is EB.
\end{lemma}

\begin{proof}
By \cite[Corollary 2]{HJRW}, there exists a commutative $C^*$-algebra $\cB \subset B(K)$  
containing $\range(\Gamma)$ and a CP map $\Theta:B(E) \to \cB \subset B(K)$ such that 
$\Theta \big|_\cS=\Gamma$. It follows from Remark \ref{rem:pos-to-eb} that $\Theta$ is EB. 
Since $\Gamma$ is the restriction of the EB map $\Theta$ to $\cS$, it is also EB.
%Alternate proof
%Let $\cB \subset B(K)$ denote the (finite-dimensional) C*-algebra generated by $\range(\Phi)$. 
%By the injectivity of $\cB$, there exists a CP map $\Theta:B(E) \to \cB$ such that $\Theta \big|_\cS=\Gamma$.
% It follows from Remark \ref{rem:pos-to-eb} that $\Theta$ is EB. Since $\Gamma$ is the 
%restriction of the EB map $\Theta$ to $\cS$, it is also EB.
\end{proof}

%{
\begin{lemma}
\label{lem:eb-map-factors}
Let $\cS \subset B(E)$ be an operator system and $\Phi:\cS \rightarrow B(H)$ be an EB map. 
There exists a commutative $C^*$-algebra $\cA$, a finite dimensional Hilbert space $K$, an 
isometry $V:H \rightarrow K$, a unital $*$-algebra homomorphism $\pi:\cA \rightarrow B(K)$ and a 
positive map $\eta:\cS \rightarrow \cA$ such that $\Phi(X) = V^*(\pi \circ \eta)(X)V$,  
for all $X \in \cS$.
\end{lemma}

\begin{proof}
 By Theorem \ref{thm:eb-extn}, there exists an EB map $\Psi:B(E) \to B(H)$ such that $\Psi\big|_{\cS}= \Phi$.  It follows from Theorem \ref{thm:ebmaps} and Corollary \ref{cor:ebmaps-on-S}, $\Phi$ can be written in Holevo form, i.e., $\Phi(X) = \sum_{j=1}^m  \phi_j(X) R_j$ for all $X \in \cS$, where the $\phi_j$'s are states on $\cS$ and the $R_j$'s are positive operators in $B(H)$.
Without loss of generality, assume that $\|\Phi\| \le 1$. 

\df{Case(i) - $\Phi$ is non-unital:} Let $\cA = \ell^\infty_{m+1}$. Define the linear maps $\gamma:\cA \rightarrow B(H)$ and $\eta:\cS \rightarrow \cA$ by $\gamma(x_1,\dots,x_{m+1}) = \sum_{j=1}^{m+1} x_jR_j$ and $\eta(X) = (\phi_1(X),\dots, \phi_m(X), 0)$, where $R_{m+1} = I_H - \sum_{j=1}^m R_j = I_H - \Phi(I_E) \in B(H)^+$. Note that $\gamma$ and $\eta$ are positive maps with $\gamma$ also being unital. In fact, by \cite[Theorem 3.11]{P}, $\gamma$ is a UCP map since the domain of 
$\gamma$ is a commutative C*-algebra.
 By Stinespring's dilation theorem  \cite[Theorem 4.1]{P}, there exists a finite dimensional Hilbert space $K$, an isometry $V:H \rightarrow K$ and a unital $*$-algebra homomorphism $\pi:\cA \rightarrow B(K)$ such that $V^* \pi(\cdot) V = \gamma(\cdot)$. 

\df{Case(ii) - $\Phi$ is unital:} Let $\cA = \ell^\infty_{m}$. Define the linear maps $\gamma:\cA \rightarrow B(H)$ and $\eta:\cS \rightarrow \cA$ by $\gamma(x_1,\dots,x_m) = \sum_{j=1}^m x_jR_j$ and $\eta(X) = (\phi_1(X),\dots,\phi_m(X))$. Note that both $\gamma$ and $\eta$ are unital positive maps. 
Arguing as above, one obtains a finite dimensional Hilbert space $K$, an isometry $V:H \rightarrow K$ and a unital $*$-algebra homomorphism $\pi:\cA \rightarrow B(K)$ such that $V^* \pi(\cdot) V = \gamma(\cdot)$. 

In both cases above,  observe that $\Phi = \gamma \circ \eta$. It follows that 
$
\Phi(X)= V^*(\pi \circ \eta)(X) V,
$
for all $X \in \cS$.
\end{proof}

\begin{proof}[Proof of Theorem~\ref{thm:eb}] To prove $(i) \implies (ii)$, observe that 
by Lemma \ref{lem:eb-map-factors}, there exists 
 a commutative $C^*$-algebra $\cA$, a finite dimensional Hilbert space $K$, an 
isometry $V:H \rightarrow K$, a unital $*$-algebra homomorphism $\pi:\cA \rightarrow B(K)$ and a 
positive map $\eta:\cS \rightarrow \cA$ such that $\Phi(X) = V^*(\pi \circ \eta)(X)V$ 
for all $X \in \cS$. Define $\Gamma := \pi \circ \eta$. Since $\pi$ is a $*$-algebra homomorphism 
and $\eta$ is a positive map with commutative range, it follows that $\Gamma$ is a 
positive map with commutative range that dilates $\Phi$. Finally, observe
 that if $\Phi$ is unital, then $\Gamma$ is too.  \\

The implication $(ii) \implies (iii)$ follows directly from Lemma \ref{lem:pos-to-eb}.\\

To prove $(iii) \implies (i)$, let $\Phi(X)= V^* \Gamma(X) V$ for all $X \in \cS$,  where $V: H \to K$ is an isometry for some finite dimensional Hilbert space $K$  and $\Gamma : \cS \to B(K)$ is an EB map with commutative range.
 It follows from Theorem \ref{thm:ebmaps} and Corollary \ref{cor:ebmaps-on-S} that $\Gamma(X) = \sum_{j=1}^{\ell} \psi_j(X)R_j$
 for some states $\psi_j$ on $\cS$ and positive operators $R_j$ in $B(K)$. Observe that for each $X \in \cS$, 
 \begin{align*}
 \Phi(X) =& V^* \left( \sum_{j=1}^{\ell} \psi_j(X)R_j \right)V
 = \sum_{j=1}^{\ell}  \psi_j(X) (V^*R_j V),
 \end{align*}  
which is again a map in the Holevo form and hence is EB. Finally, if $\Gamma$ is unital, 
then $\Phi$ is too. 
\end{proof}
%
%\begin{remark}
%Note that Theorem \ref{thm:eb} can also be deduced from Naimark's dilation 
%theorem (\cite[Theorem 4.6]{P}). 
%\end{remark}
%}

\section{Maximal UEB Dilations}
\label{S:MaxUEB}
%As an application of the EB extension result, we give a characterization of the maximal UEB maps defined on an operator system.
In this section we prove Theorems \ref{thm:max-iff-stdform} and \ref{thm:maxdil-of-UEB}. 
The proof of Theorem \ref{thm:max-iff-stdform} makes use of the following %Proposition and
 Lemmas, which 
contain some key properties of maximal UEB maps defined on operator systems.
% property of maximal UEB maps along with the EB extension result.

\begin{lemma}
\label{lem:comm-range}
Let $\cS \subset B(E)$ be an operator system and $\Phi:\cS \rightarrow B(H)$ be 
a UEB map. If $\Phi$ has commutative range, then 
\begin{equation}
\label{eq:form}
\Phi(X)=\sum_{i=1}^k\phi_i(X)P_i,
\end{equation}
where the $\phi_i$'s are distinct states on $\cS$ and the $P_i$'s 
are mutually orthogonal projections in $B(H)$ such that $\sum_{i=1}^k P_i = I_H$.
\end{lemma}

\begin{proof}
 Let $\cA= C^*(\range (\Phi))\subset B(H)$. There 
 exists a unital $*$-algebra isomorphism $\pi: \cA \to \ell_k^\infty$ for some $k$.
Consider $~\pi\circ \Phi: \cS \to \ell_k^\infty$. 
%will be of the form 
For each $X\in \cS$, $\pi\circ \Phi(X)=\sum_{i=1}^k\lambda_{i,X}e_i$,
where $\lambda_{i,X}$'s are scalars (depending on $X$) and $\{e_1, \dots, e_k\}$ is the standard basis of $\ell_k^\infty$.
For $1\leq i \leq k$, define $\phi_i: \cS \to \C$ via $\phi_i(X)=\lambda_{i,X}$.
Since $\pi \circ \Phi$ is a unital positive map, it follows that $\phi_i$'s are states on $\cS$.
%Observe that $\phi_i$'s are states on $\cS$ (since $\pi \circ \phi$ is a unital positive map) 
 Indeed
\begin{equation*}
\label{eq:form}
\Phi(X)=\sum_{i=1}^k\phi_i(X)P_i,
\end{equation*}
where $P_i=\pi^{-1}(\{e_i\})$.
Note that the $P_i$'s are mutually orthogonal projections in $B(H)$ such that $\sum_{i=1}^kP_i=I_H$. 
Without loss of generality, one can assume that the $\phi_i$'s are distinct states on $\cS$ because if 
two $\phi_i$'s are the same, we can sum the corresponding projections,
 i.e., the $P_i$'s together, and obtain a single projection. 
\end{proof}

\begin{lemma} 
\label{lem:maximal-imp-commrang}
Let $\cS \subset B(E)$ be an operator system. 
If $\Phi: \cS \to B(H)$ is a maximal UEB map, then $\range(\Phi)$ is 
commutative.
\end{lemma}

\begin{proof}
Since $\Phi$ is a UEB map,
% it follows from Theorem \ref{thm:eb-extn} that there
% exists a UEB map $\Psi:B(E) \to B(H)$ such that $\Psi\big|_\cS=\Phi$.
it follows from part (iii) of 
Theorem \ref{thm:eb} that there exists a finite dimensional Hilbert space $K$, a
 UEB map 
$\Gamma:\cS  \to B(K)$ with commutative range and an isometry $V:H \rightarrow K$ 
such that $V^*\Gamma(X) V=\Phi(X)$,
for all $X \in \cS$. Since $\Phi$ is a maximal UEB map, $\Gamma$ is
a trivial UEB dilation of $\Phi$. Thus it follows from Remark \ref{rem:alt-defn} that 
the subspace $VH$ is an invariant subspace 
for $\Gamma(X)$ for all $X\in \cS$. 
Since $VV^*$ is the projection of $K$ onto $\range(V)$, it follows that 
\begin{align*}
\Phi(X)\Phi(Y)(a) &= V^*\Gamma(X)VV^*\Gamma(Y)V(a)\\
&= V^*\Gamma(X)\Gamma(Y)V(a)\\
&= V^*\Gamma(Y)\Gamma(X)V(a)\\
&= V^*\Gamma(Y)VV^*\Gamma(X)V(a)\\
&= \Phi(Y)\Phi(X)(a),\quad \text{ for all $a \in H$.} \qedhere
\end{align*} 
\end{proof}

\begin{lemma}
 \label{lem:max} Let $\cS \subset B(E)$ be an
 operator system and let $\Phi : \cS \to B(H)$ be a UEB map defined by $\Phi(X)=\sum_{i=1}^k \phi_i(X)P_i$,
 where $\phi_i$ are states defined on $\cS$ and $P_i \in B(H)$ are mutually orthogonal 
projections satisfying $\sum_{i=1}^k P_i = I_H$. If $\Phi$ is a maximal UEB map, 
then each $\phi_i$ is a linear 
extremal state.
\end{lemma}

\begin{proof}
%We will show for the case $k=2$. General case will follow using the same argument.

Fix $i \in \{1,2,\dots,k\}$. Suppose that $\sigma, \tau$ are states on $\cS$ such that 
\[
\phi_i= t \sigma + (1-t) \tau,
\]
for some $t \in(0,1)$. It suffices to show that $\sigma = \tau$.
Let 
$F= \range (P_i) \subset H$ and $G= F^\perp \subset H$. 
%Indeed $H= F \oplus G$.
Let $L:= F \oplus F \oplus G$.
Define $V:H \rightarrow L$ by 
\begin{equation*}
V(x+y)= (\sqrt{t}x, \sqrt{(1-t)} x, y),
\end{equation*}
for all $x\in F$ and $y\in G$. It is easily seen that $V$ is an isometry. Let 
$\Psi : \cS \to B(L)$ be defined by 
\begin{equation*}
\Psi(X)(x, y, z)= \bigg(\sigma(X)x, \tau(X)y, \sum_{j\neq i} \phi_j(X)P_j(z)\bigg),
\end{equation*}
for all $X\in \cS, x,y \in F$ and $z \in G$. Observe that $\Psi$ is unital.
Define the coordinate projections $Q_1, Q_2, Q_3$ on $L$ by
 $Q_1(x,y,z)=(x,0,0), Q_2(x,y,z)=(0,y,0)$ and $Q_3(x,y,z)=(0,0,z)$.
One sees that $\Psi(X)= \sigma(X)Q_1+\tau(X)Q_2+ \left(\sum_{j\neq i}\phi_j(X)P_j \right)Q_3$, 
which is in Holevo form for $\Psi$. 
%Hence $\Psi$ is a UEB map. 
By Corollary \ref{cor:ebmaps-on-S}, it follows that $\Psi$ is a UEB map. 
Observe that 
\begin{align*}
\langle V^*(u, v, w), (x+y)\rangle =& \langle (u, v, w), V(x+y)\rangle\\
=& \langle (u, v, w), (\sqrt{t}x, \sqrt{(1-t)} x, y)\rangle\\
=& \sqrt{t} \langle u,x\rangle+ \sqrt{(1-t)} \langle v, x\rangle + \langle w, y\rangle\\
=& \langle \sqrt{t} u + \sqrt{(1-t)} v, x\rangle+ \langle w, y\rangle\\
=& \langle \sqrt{t} u + \sqrt{(1-t)} v + w, x+y \rangle,
\end{align*}
for all $u,v,x \in F$ and $w,y \in G$. Thus $V^*$ is defined by 
\[
V^*(u, v, w) := \sqrt{t} u + \sqrt{(1-t)} v + w.
\]
%for all $u,v \in F$ and $w\in G$. 
For each $X\in \cS$, $x \in F$ and $y\in G$, it follows that 
\begin{align*}
V^* \Psi(X)V(x+y) &= V^*\Psi(X)(\sqrt{t}x, \sqrt{(1-t)} x, y)\\
& = V^*\bigg( \sqrt{t} \sigma(X)x, \sqrt{(1-t)} \tau(X)x, \sum_{j\neq i} \phi_j(X)P_j(y)\bigg)\\
& = t \sigma(X)x+ (1-t) \tau(X) x+ \sum_{j\neq i} \phi_j(X)P_j(y)\\
& = \phi_i(X) x+ \sum_{j\neq i} \phi_j(X)P_j(y)\\
& = \sum_{j=1}^k \phi_j(X)P_j(x+y)\\
& = \Phi(X)(x+y).
\end{align*}
Hence $\Psi$ is a UEB dilation of $\Phi$. 
By hypothesis, 
$\Psi(X)VH \subset VH$ for all $X\in \cS$.
 Choosing $0 \neq x \in F$, it follows that 
\begin{align*}
\big(\sqrt{t}\sigma(X)x, \sqrt{(1-t)} \tau(X) x,0\big) &= \Psi(X)V(x) \\
&= V(z + w) = (\sqrt{t}z, \sqrt{(1-t)}z, w)
\end{align*}
for some $z\in F$ and $w\in G$. 
Hence $\sigma(X)x=\tau(X)x$ for 
each $X \in \cS$. This in turn implies that
 $\sigma = \tau$ and the 
proof is complete. \qedhere
\end{proof}

\begin{lemma}
\label{lem:fin-max} 
Let $\cS \subset B(E)$ be an operator system 
%$H$ be a finite dimensional Hilbert space 
and let $\Phi: \cS \to B(H)$ be a UEB map. If every UEB dilation $\Psi: \cS \to B(L)$ of $\Phi$ is trivial  whenever $L$ is finite dimensional,
 then $\Phi$ is a maximal UEB map. 
\begin{proof}
Let $\Psi: \cS \to B(K)$ be a UEB dilation of $\Phi$ where
 $K$ is an infinite dimensional separable Hilbert space. Let the isometry 
$V:H \rightarrow K$ be such that $V^*\Psi(X)V=\Phi(X)$ for 
all $X\in \cS$.  Let $\{e_1, \dots, e_n\}$ denote 
an orthonormal basis for $VH$. Extend it to an orthonormal
 basis $\{e_i \,:\, i \in \N\}$ for the Hilbert space $K$. 
%{
 For each $m \ge n$, let $K_m:=\operatorname{span}\{e_1, \dots, e_m\}$, 
$W_m:K_m \rightarrow K$  denote the inclusion map and 
$P_m$  denote the orthogonal projection of $K$ onto $K_m$. 
Observe that $P_m = W_m^*$. Define $V_m = P_m \circ V$ and 
 $\Psi_m : \cS \to B(K_m)$ to be the compression 
of $\Psi$ to $K_m$, i.e., 
$\Psi_m(X)=  P_m \Psi(X)\big|_{K_m} = W_m^* \Psi(X) W_m$. 
Let $k \in \N$ and $Z \in (M_k \otimes \cS)^+$. Since $\Psi$ is an EB map, 
$(\tti_k \otimes \Psi)(Z) \in \overline{M_k^+ \otimes B(K)^+}$. It follows that 
$
(\tti_k \otimes \Psi_m)(Z) = (\tti_k \otimes W_m)^*( (\tti_k \otimes \Psi)(Z)) (\tti_k \otimes W_m) \in \overline{M_k^+ \otimes B(K_m)^+} = M_k^+ \otimes B(K_m)^+.
$
Thus $\Psi_m : \cS \to B(K_m)$ is a UEB map. Moreover, $V_m$ is an isometry and 
$V_m^*\Psi_m(X)V_m = \Phi(X)$ for all $X \in \cS$, i.e., $\Psi_m$ is a UEB dilation of $\Phi$.
%We will first show that $\Psi_m$ is a UEB dilation of $\Phi$ for each $m\geq n$. 
Let $X\in \cS$. 
%%$V:H \to K$ is an isometry such that $V^*\Psi(X)V=\Phi(X)$ for all $X \in \cS$. 
%Observe that $V_m:H\to K_m$ is an isometry, for each $m \ge n$. Also for $a \in H$, it follows that 
%\begin{align*}
%V_m^*\Psi_m(X)V_m(a) &= V^*P_m\Psi(X)P_mV(a)\\
%&= V^*P_m\Psi(X)V(a)\\
%&= V^*(VV^*)P_m\Psi(X)V(a)\\
%&= V^*(VV^*) \Psi(X)V(a)\\
%%&= V^* \Psi(X)V(a)\\
%&= \Phi(X)(a).
%\end{align*}
% Thus $\Psi_m$ is a UEB dilation of $\Phi$, for each $m \ge n$. 
Since $K_m$ is finite dimensional, by hypothesis,
 it follows that $\Psi_m$ is a trivial UEB dilation of $\Phi$. By Remark \ref{rem:alt-defn}, 
one gets that $\Psi_m(X)V_mH \subset V_m H = VH$. Finally, for $x \in H$, 
\begin{align*}
\Psi(X)(Vx)& = \lim_{m \rightarrow \infty} P_m \Psi(X) P_m (Vx) = \lim_{m \rightarrow \infty} W_m^* \Psi(X) W_m P_m (Vx) \\
& =  \lim_{m \rightarrow \infty} \Psi_m(X)V_m(x) \in  VH.  \qedhere
\end{align*}
\end{proof}
\end{lemma}

\begin{remark}
It is to be noted that the definition of an EB map adopted here in terms of separability in the norm-closure sense, is not the usual way it is defined in the literature. If one works with a weak*-continuous CP map (or equivalently, a normal CP map), say $\Phi:B(L) \rightarrow B(M)$, then there is a way of defining an EB map that is along the lines of the usual definition namely: $\Phi:B(L) \rightarrow B(M)$ is EB if for every $k \in \N$,  $\Psi:\fT(M) \rightarrow \fT(L)$ maps positive matrices in $M_k \otimes \fT(M)$ to separable matrices in $M_k \otimes \fT(L)$, where $\fT(L)$ denotes the space of trace-class operators on the Hilbert space $L$, and $\Psi$ is the unique map whose (Banach space) adjoint $\Psi^*$ equals $\Phi$. This certainly is a generalization of the usual definition of an EB map from the finite to the infinite dimensional setting. But since here we mainly work on unital maps defined only on operator systems, defining the  "entanglement breaking" property using this type of duality becomes a challenge. 
\end{remark}

We would like to emphasize that the above remark is only relevant 
when one of the Hilbert spaces $L$ or $M$ is infinite dimensional and so 
does not impact the results here.

\begin{proof}[Proof of Theorem~\ref{thm:max-iff-stdform}]
%$(i) \Rightarrow (ii)$: 
%Since $\Phi$ is a maximal UEB map, it follows from Lemma 
% \ref{lem:maximal-imp-commrang} that
% the map $\Phi$ has commutative range. Let $\cA= C^*(\range (\Phi))\subset B(H)$. There 
% exists a unital $*$-algebra isomorphism $\pi: \cA \to \ell_k^\infty$ for some $k$.
%Consider $~\pi\circ \Phi: \cS \to \ell_k^\infty$. 
%%will be of the form 
%For each $X\in \cS$, $\pi\circ \Phi(X)=\sum_{i=1}^k\lambda_{i,X}e_i$,
%where $\lambda_{i,X}$'s are scalars (depending on $X$) and $\{e_1, \dots, e_k\}$ is the standard basis of $\ell_k^\infty$.
%For $1\leq i \leq k$, define $\phi_i: \cS \to \C$ via $\phi_i(X)=\lambda_{i,X}$.
%Since $\pi \circ \Phi$ is a unital positive map, it follows that $\phi_i$'s are states on $\cS$.
%%Observe that $\phi_i$'s are states on $\cS$ (since $\pi \circ \phi$ is a unital positive map) 
% Indeed
%\begin{equation}
%\label{eq:form}
%\Phi(X)=\sum_{i=1}^k\phi_i(X)P_i,
%\end{equation}
%where $P_i=\pi^{-1}(\{e_i\})$.
%Note that $P_i$'s are mutually orthogonal projections in $B(H)$ such that $\sum_{i=1}^kP_i=I_H$.

 (i)$ \implies$ (ii): Since $\Phi$ is a maximal UEB map, it follows from 
Lemma \ref{lem:maximal-imp-commrang} that $\Phi$ has commutative range. 
By applying Lemma \ref{lem:comm-range}, one gets that 
$\Phi(X)=\sum_{i=1}^k \phi_i(X)P_i$,
 where the $\phi_i$'s are distinct states defined on $\cS$ and $P_i \in B(H)$ 
are mutually orthogonal projections satisfying $\sum_{i=1}^k P_i = I_H$. That
 the $\phi_i$'s are linear extremal states on $\cS$ follows from 
an application of Lemma \ref{lem:max}.
 
To prove the implication (ii) $\implies$ (i),  %{ 
let $\Psi$ be an arbitrary UEB dilation of $\Phi;$  that is, $\Psi:\cS \to B(K)$ is a UEB map such that there exists an isometry $V:H \to K$ satisfying
$\Phi(X)= V^*\Psi(X)V,$ for all $X\in \cS$. 
By Lemma \ref{lem:fin-max}, one can assume that $K$ is finite dimensional. 
%}
It suffices to show that $\Psi(X)Vx \in VH$ for all $x\in H$ and $X\in \cS$. If $x=0$, then there is nothing to prove. Suppose that $x\neq0$.  We first consider the case $x\in \operatorname{Range}(P_j)$ for some $j$. For such an $x$, it follows that $\Phi(X)x= \phi_j(X)x$ and hence $V^* \Psi(X)Vx=\phi_j(X)x$.

Let a Holevo form (see Theorem \ref{thm:ebmaps} 
and Corollary \ref{cor:ebmaps-on-S})
for $\Psi$ be given by $\Psi(X)= \sum_{i=1}^\ell \psi_i(X)R_i$, where the $\psi_i$'s are distinct states and the $R_i$'s are positive matrices with $\sum_{i=1}^\ell R_i=I_K$. 
Observe that 
 \begin{align}
\label{eq:simplify}
\notag
\phi_j(X) \|x\|^2 =  \langle\phi_j(X)x,x\rangle & = \langle \Phi(X)x, x \rangle  = \left\langle \left (\sum_{i=1}^\ell V^*  \psi_i(X)R_i V \right)x,x \right \rangle\\
 %=& \left \langle \sum_{i=1}^k \psi_i(X)R_iVx,Vx \right \rangle\\
 =& \sum_{i=1}^\ell \psi_i(X) \left \langle R_iVx,Vx \right \rangle.
 \end{align}
%{ 
Thus, $\phi_j(X)= \sum_{i=1}^\ell \alpha_i \psi_i(X)$, where $\alpha_i = \frac{1}{\|x\|^2} \langle R_iVx,Vx \rangle \ge 0$. 
Note that $\alpha_i$'s are independent of $X$. Also, since $V$ is an isometry and $\sum_{i=1}^\ell R_i=I_K$,
 it follows that 
\begin{equation}
\label{eq:convexcomb}
\sum_{i=1}^\ell \alpha_i = \frac{1}{\|x\|^2} \sum_{i=1}^\ell \langle R_iVx,Vx \rangle = \frac{1}{\|x\|^2} \left \langle \left( \sum_{i=1}^\ell R_i \right) Vx, Vx \right \rangle = 1.
\end{equation}
Combining equations \eqref{eq:simplify} and \eqref{eq:convexcomb}, it follows that 
 $\phi_j$ is a convex combination of the $\psi_i$'s.
%}
 Since the $\psi_i$'s are distinct and $\phi_j$ is a linear extremal state, it follows that there exists an index $i_0$ such that $\alpha_{i_0} = 1$, $\phi_j= \psi_{i_0}$ and $\alpha_i = \frac {1}{\|x\|^2} \langle R_iVx,Vx\rangle=0$ for all $i\neq i_0$. Since the $R_i$'s are positive, it follows that $R_iVx=0$ for all $i\neq i_0$, implying $Vx = \sum_{i=1}^\ell R_iVx = R_{i_0}Vx$. 
 Thus
 \begin{align*}
 \Psi(X)Vx=& \psi_{i_0}(X)R_{i_0}Vx= \psi_{i_0}(X)Vx \in VH.
 \end{align*}
%It has been shown that if $x \in \operatorname{Range}(P_j)$ for some $j$, then $\psi(X)Vx \in VH$. 
Now using the fact  that the $P_i$'s are mutually orthogonal projections satisfying $\sum_{i=1}^k P_i=I_H$ and $\oplus_{i=1}^k \range(P_i)=H$, it follows that $\Psi(X)Vx\in V(H)$ for all $x\in H$, and the proof is complete. \qedhere
\end{proof}

Our next task is to prove Theorem \ref{thm:maxdil-of-UEB}, before which 
we prove the following crucial observation concerning a Holevo form of a UEB map. 

\begin{lemma}
\label{lem:extremal}
Let $\cS \subset B(E)$ be an operator system. If $\Phi:\cS \rightarrow B(H)$ is a UEB map, 
then $\Phi(X) = \sum_{j=1}^k \phi_j(X)S_j$ for each $X \in \cS$, where $\phi_j$ are distinct 
linear extremal states on $\cS$ and $S_j$ are positive operators in $B(H)$ satisfying 
$\sum_{j=1}^k S_j= I_H$.
\end{lemma}
\begin{proof}
Let a Holevo form for $\Phi$ be given by 
\[
\Phi(X)= \sum_{i=1}^r \phi_i(X)R_i,
\]
where the $\phi_i$'s are states on $\cS$ and the $R_i$'s are
 positive operators in $B(H)$ such that $\sum_{i=1}^r R_i= I_H$. Observe that 
the set of all states on $\cS$ is a compact and convex subset of the dual of $\cS$. 
A well known consequence of Caratheodory's theorem (See \cite[Theorem 16.1.8]{DD} and 
\cite[Corollary 16.1.9]{DD}) and the Krein-Milman theorem is that  in a finite dimensional 
 topological vector space, a compact convex set $\cC$ equals the convex hull 
of its extreme points. Using this fact here, each $\phi_i$ can be written as a convex 
combination of linear extremal states, i.e.,
 \[
 \phi_i(X)= t_{i,1}\varphi_{i,1}(X)+ \dots + t_{i,\ell_i}\varphi_{i,\ell_i}(X),
 \]
where $t_{i,j}\in [0,1]$ with $\sum_{j=1}^{\ell_i} t_{i,j}=1$ and the $\varphi_{i,j}$'s 
are linear extremal states on $\cS$.
It follows that 
\begin{align*}
\Phi(X) &= \sum_{i=1}^r \sum_{j=1}^{\ell_i} t_{i,j}\varphi_{i,j}(X)R_i =  \sum_{i=1}^r \sum_{j=1}^{\ell_i} \varphi_{i,j}(X)(t_{i,j}R_i),
\end{align*}
for all $X \in \cS$.
Note that by combining suitable terms in the above sum, it
 can be rewritten in the desired form. This completes the proof.
\end{proof}

\begin{proof}[Proof of Theorem~\ref{thm:maxdil-of-UEB}]
Using Theorem \ref{thm:eb-extn} and the Holevo form (see Theorem \ref{thm:ebmaps}  
and Corollary \ref{cor:ebmaps-on-S}) for EB maps, we have 
\begin{equation}
\label{eq:hol-form}
\Phi(X)= \sum_{i=1}^r \phi_i(X)R_i,
\end{equation}
for all $X \in \cS$, where the $\phi_i$'s are distinct states on $\cS$ and the $R_i$'s are
 positive operators in $B(H)$ such that $\sum_{i=1}^r R_i= I_H$.
%It is a well known consequence of Caratheodory's theorem (See \cite[Theorem 16.1.8]{DD} and 
%\cite[Corollary 16.1.9]{DD}) and the Krein-Milman theorem that  in a finite dimensional 
% topological vector space, a compact convex set $\cC$ equals the convex hull 
%of its extreme points. Using this fact, each $\phi_i$ can be written as convex 
%combination of linear extremal states, i.e.,
% \[
% \phi_i(X)= t_{i,1}\varphi_{i,1}(X)+ \dots + t_{i,\ell_i}\varphi_{i,\ell_i}(X),
% \]
%where $t_{i,j}\in [0,1]$ with $\sum_{j=1}^{\ell_i} t_{i,j}=1$ and $\varphi_{i,j}$'s are linear extremal states.
%Now \begin{align*}
%\Phi(X)=& \sum_{i=1}^r \sum_{j=1}^{\ell_i} t_{i,j}\varphi_{i,j}(X)R_i\\
%=& \sum_{i=1}^r \sum_{j=1}^{\ell_i} \varphi_{i,j}(X)(t_{i,j}R_i).===
%\end{align*} 
By Lemma \ref{lem:extremal}, there is no loss of generality in assuming that the
 $\phi_i$'s in equation \eqref{eq:hol-form} are distinct linear extremal states on $\cS$. 
By Lemma \ref{lem:eb-map-factors}, there exists a finite dimensional Hilbert space $K$, an 
isometry $V:H \rightarrow K$, $r \in \N$, a unital $*$-algebra homomorphism 
$\pi:\ell^{\infty}_r \rightarrow B(K)$ and a positive map $\eta:\cS \rightarrow \ell^{\infty}_r$ 
such that 
\[
\Phi(X) = V^*(\pi \circ \eta)(X)V,
\]
for all $X \in \cS$. As observed in the proof of Lemma \ref{lem:eb-map-factors}, 
the unitality of $\Phi$ implies the unitality of $\eta$. 
%Repeating as in the proof of Theorem \ref{thm:eb} (and retaining the same notations
% used there),
% it is easily seen that $\Phi$ dilates to a map of the form $\pi \circ \eta: \cS \to B(K)$,
%%( where $\pi$ and $\eta$ are defined as in the proof of Theorem \ref{thm:eb})
%i.e.,  there exists an isometry $V: H \to K$ such that 
%\[
%V^*(\pi \circ \eta)(X)V= \Phi(X)
%\]
%for all $X\in \cS$. 
Note that $\pi(x_1, \dots, x_r)= \sum_{i=1}^r x_i\pi(e_i)$, where $\{e_1, \dots, e_r\}$ is 
the standard basis of $\ell_r^\infty$. Also 
$(\pi \circ \eta)(X)= \pi\big(\sum_{i=1}^r \phi_i(X)e_i\big)=  \sum_{i=1}^r \phi_i(X)\pi(e_i)$, 
for all $X \in \cS$.
Since $\pi$ is a unital $*$-algebra 
homomorphism, it follows that the $\pi(e_i)$'s are mutually orthogonal projections and
 $\sum_{i=1}^r \pi(e_i)= I_K$. Define $\Psi:= \pi \circ \eta$. It follows from 
Theorem \ref{thm:max-iff-stdform} that $\Psi$ is maximal. This completes the proof.  \qedhere
%}
\end{proof}

\section{$C^*$-extreme UEB maps on Operator Systems}
\label{S:C*-ext}
%{
In this section we prove Theorem \ref{thm:max-iff-C*ext} and Corollary \ref{cor:extn}. 
The proofs use techniques from \cite{FM93}, \cite{FM97}, \cite{FZ} and \cite{Z}. We also 
need the following Lemmas, the first of which contains an equivalent definition of a $C^*$-extreme UEB map and 
the second one contains a description of UEB$(\cS, B(H))$ in terms of $C^*$-convex combinations. 
\begin{lemma}
\label{lem:alt-defn}
Let $\cS \subset B(E)$ be an operator system and $\Phi:\cS \rightarrow B(H)$ be a UEB map. The following statements are equivalent.
\begin{itemize}
\item[(i)] $\Phi$ is a $C^*$-extreme point of UEB$(\cS, B(H))$.
\item[(ii)] If $\Phi = \sum_{i=1}^2 T_i^* \Phi_i T_i$, for some invertible operators $T_1, T_2 \in B(H)$ satisfying 
$T_1^*T_1 + T_2^*T_2 = I_H$, then there exist unitaries $U_1, U_2 \in B(H)$ such that 
$\Phi_i(X)=U_i^*\Phi(X) U_i$, $i= 1, 2$.
\end{itemize}
\end{lemma}
Lemma~\ref{lem:alt-defn} is essentially the operator system version of \cite[Proposition 3.2]{BDMS}. The proof given 
there works equally well for the operator system setting too.
%}
\begin{lemma}
\label{lem:ueb}
Let $\cS \subset B(E)$ be an operator system, $\mathcal K$ denote the set 
\begin{align*}
\{\Psi:\cS \rightarrow B(H)& : \Psi(X) = g(X)I_H, \, g  \text{ is a linear extremal state on } \cS \},
\end{align*}
$\cE$ denote the $C^*$-convex hull of $\mathcal K$ 
% i.e., the set of all $C^*$-convex combinations of elements of $\cK$
 and
$\mathcal E_+$ denote the set of all positive $C^*$-convex combinations 
of elements of $\cK$. The following statements hold. 
\begin{enumerate}[(i)]
\item $\cE_+ = \cE = $ UEB$(\cS, B(H))$.
\item \label{eq:pos} If $\Gamma \in \mathcal E$ is given by $\Gamma(X) = 
\sum_{i=1}^m T_i^*(g_i(X)I_H)T_i \in \mathcal E$, 
then $\Gamma \in \mathcal E_+$ and there exists $B_i \in B(H)^+$ such that $\sum_{i=1}^m B_i^2 = I_H$ and 
\[
\Gamma(X) = \sum_{i=1}^m B_i(g_i(X)I_H)B_i,
\] 
for all $X \in \cS$.
\end{enumerate}
\end{lemma}
\begin{proof}
Evidently, $\cE_+ \subset \cE \subset $ UEB$(\cS, B(H))$. To complete the proof of part $(i)$, 
 it suffices to show that  UEB$(\cS, B(H))$ $\subset \cE_+$. To this end, let $\Phi \in$ UEB$(\cS, B(H)).$ 
By Lemma \ref{lem:extremal}, it follows that for each $X \in \cS$, $\Phi(X) = \sum_{j=1}^k \phi_j(X)S_j$,  
where $\phi_j$ are distinct linear extremal states on $\cS$ and $S_j$ are 
positive operators in $B(H)$ satisfying $\sum_{j=1}^k S_j= I_H$. 
Rewrite
\[
\Phi(X) = \sum_{j=1}^k \phi_j(X)S_j = \sum_{j=1}^k \sqrt{S_j} (\phi_j(X)I_H) \sqrt{S_j}. 
\]
Thus $\Phi \in \mathcal E_+$. To prove part (ii), write $T_i$ in its polar decomposition, i.e., $T_i = U_iB_i$, where $U_i$ is a unitary operator and 
$B_i = \sqrt{T_i^*T_i} \in B(H)^+$. Note that $\sum_{i=1}^m B_i^2 = \sum_{i=1}^m T_i^*T_i = I_H$. 
By part $(i)$, it follows that $\Gamma \in \mathcal E_+$. Also, for each $X \in \cS$, 
\begin{align*}
\Gamma(X) &= \sum_{i=1}^m T_i^*(g_i(X)I_H)T_i = \sum_{i=1}^m B_iU_i^*(g_i(X)I_H)U_iB_i \\
& = \sum_{i=1}^m B_i(g_i(X)I_H)B_i. \qedhere
\end{align*}
\end{proof}

\begin{proof}[Proof of Theorem~\ref{thm:max-iff-C*ext}]
$(i) \Rightarrow (ii):$ Using the alternate definition of a $C^*$-extreme UEB map given in
 Lemma \ref{lem:alt-defn}, let $\Phi(X)= T_1^*\Phi_1(X)T_1+T_2^*\Phi_2(X)T_2$ for all 
$X\in \cS$, where $\Phi_1, \Phi_2:\cS \to B(H)$
 are UEB maps and $T_1, T_2\in B(H)$ are invertible operators such that 
$\sum_{i=1}^2T_i^*T_i=I_H$. It suffices to show that $\Phi_1$ and $\Phi_2$ are 
unitarily equivalent to $\Phi$. 

Define an isometry $V: H \to H \oplus H$ via $V(x)= (T_1(x), T_2(x))$ and a linear
 map $\Psi:\cS \to B(H\oplus H)$ by
\[
\Psi(X)= \begin{bmatrix}
\Phi_1(X) & 0\\
0 & \Phi_2(X)
\end{bmatrix}, 
\] 
for all $X\in \cS$. Observe that $\Psi$ being a direct sum of UEB maps, is a UEB map. 
In fact, it is a UEB dilation of $\Phi$ since $V^*\Psi(X)V=\Phi(X)$ for all $X\in \cS$.
By hypothesis, it follows that $\Psi$ is a trivial UEB dilation of
 $\Phi$, that is, $\Psi(X)V= V\Phi(X)$ for all $X\in \cS$. Equivalently,
\[
\begin{bmatrix}
\Phi_1(X) & 0\\
0 & \Phi_2(X)
\end{bmatrix} \begin{bmatrix}
T_1\\
T_2
\end{bmatrix} = \begin{bmatrix}
T_1\\
T_2
\end{bmatrix} \Phi(X)
\]
for all $X \in \cS$. Consequently, we have
%\[
%\Phi_1(X)T_1 =T_1 \Phi(X)\, \text{ and } \, \Phi_2(X)T_2 =T_2\Phi(X).
%\]
%Equivalently,
\[
\Phi_1(X) = T_1 \Phi(X)T_1^{-1} \, \text{ and }\,  \Phi_2(X) = T_2 \Phi(X)T_2^{-1}
\]
for all $X\in\cS$.

Using the polar decomposition of the invertible operator $T_i^*$, $i=1,2$, there exists a unitary
 operator $W_i\in B(H)$ and a
 positive invertible operator $P_i\in B(H)$ given by $\sqrt{T_iT_i^*}$  such that 
$T_i=P_iW_i$, $i=1,2$. 
%Thus $\Phi_1(X)= P_1W_1 \Phi(X)W_1^*P_1^{-1}$.
 Hence
\begin{equation}
\label{eq:def-of-phi}
\Phi(X)= W_i^*P_i^{-1}\Phi_i(X)P_iW_i,
\end{equation}
for $i = 1, 2.$ Since $\Psi$ is a trivial UEB dilation of $\Phi$, it follows that 
\begin{equation}
\label{eq:tri-dil}
\Psi(X)VV^*= VV^*\Psi(X) ,
\end{equation}
for all $X\in \cS$, (See Remark \ref{rem:alt-defn}).
Applying the definition of $\Psi$ and $VV^*$ in Equation \eqref{eq:tri-dil}, one obtains 
\begin{align*}
\Phi_i(X)T_iT_i^* =& T_iT_i^*\Phi_i(X),
%\Phi_2(X)T_2T_2^* =& T_2T_2^*\Phi_2(X).
\end{align*}
for $i = 1,2$. Note that $T_iT_i^*$ is a positive operator and $\Phi_i(X)$ commutes with $T_iT_i^*$. 
Hence
\begin{equation}
\label{eq:comm-equality}
\Phi_i(X)P_i= P_i\Phi_i(X),
\end{equation}
for each $X \in \cS$ and $i = 1,2$.
 
It follows from equations \eqref{eq:def-of-phi} and \eqref{eq:comm-equality} that 
 \begin{align*}
 \Phi(X)= W_i^* P_i^{-1}P_i \Phi_i(X)W_i = W_i^*\Phi_i(X)W_i,
 \end{align*}
for each $X \in \cS$ and $i = 1,2$. Since $W_1$ and $W_2$ are unitaries, the proof 
is complete.

The following proof of the implication $(ii) \Rightarrow (i)$ is adapted from \cite[Theorem 4.1]{FM93}. 
Recall the notations $\mathcal K$, $\mathcal E$ and $\mathcal E_+$ from Lemma \ref{lem:ueb}.
Let $\Phi:\cS \rightarrow B(H)$ be a $C^*$-extreme UEB map. 
Using Lemma \ref{lem:ueb}, write $\Phi(X)=\sum_{i=1}^m T_i^* \Psi_i(X)T_i$, where $\Psi_i(X)=g_i(X)I_H$ 
for all $X \in \cS$, the $g_i$'s are linear extremal states on $\cS$, and $\sum_{i=1}^m T_i^*T_i=I_H$. 
By part \eqref{eq:pos} of Lemma \ref{lem:ueb}, there is no loss of generality in assuming that the $T_i$'s are positive. 
Let $m$ be the least number of coefficients required to represent $\Phi$ as a positive $C^*$-convex 
combination, i.e., as a $C^*$-convex combination 
of elements of $\mathcal K$ with positive coefficients (or equivalently to represent $\Phi$ as an element 
of $\mathcal E_+$). If $m=1$, then there is nothing to prove, due to 
Theorem \ref{thm:max-iff-stdform}. So assume $m\geq 2$. Note that there must exist an $i$ such that $\| T_i\|=1$. 
Otherwise $\|T_i\| < 1$ for all $i$ and, by modifying Technique-A in \cite{FM93} to
 our current setting, one can rewrite $\Phi$ as a proper $C^*$-convex 
combination of some $\Gamma_j \in \mathcal E$ with the additional property that
each $\Gamma_j$ is a $C^*$-convex combination of less than $m$ of the $\Psi_i$'s. 
Since $\Phi$ is $C^*$- extreme, each $\Gamma_j$ is unitarily
 equivalent to $\Phi$. An application of part \eqref{eq:pos} of Lemma \ref{lem:ueb} 
 implies that $\Phi$ is a positive $C^*$-convex combination of fewer 
than $m$ of the $\Psi_i$'s, which contradicts the minimality of $m$. 
Thus at least one of the $T_i$'s has unit norm. Without loss of generality, assume that $\|T_1\|=1$. 
%By expressing each operator $T_i$ in its polar decomposition 
%$T_i=U_iB_i$, where $U_i \in B(H)$ is unitary and $B_i = \sqrt{T_i^*T_i} \in B(H)$ is 
%positive, it follows that 
%\begin{align*}
%\Phi(X)=& \sum_{i=1}^m B_iU_i^*\Psi_i(X)U_iB_i =  \sum_{i=1}^m B_i \Psi_i(X)B_i.
%\end{align*}
Due to the unitality of $\Phi$, note also that $\sum_{i=1}^m T_i^2=I_H$.
 Since $T_1\geq 0$ and $\|T_1\|=1$, 
it follows that there exists a unitary $U\in B(H)$ such that 
\[
U^*T_1U = \begin{pmatrix}
I &  0\\ 
0 &   Y_1
\end{pmatrix}
\]
where $Y_1$ is a diagonal matrix satisfying 
$Y_1\geq 0$, $\|Y_1\|<1$ and $I$ is 
the identity operator of suitable size.
Since $T_i\geq 0$ for all $i\geq 2$ and $\sum_{i=1}^m  T_i^2=I_H$, 
\[
U^*T_iU= \begin{pmatrix}
0 & \\ 
& Y_i
\end{pmatrix},  \text{ with } Y_i \ge 0 \text{ for all } i \geq 2.
\]
Let $W_i = U^*T_iU$ for each $i \in \{1,2,\dots,m\}$. 
Observe that $W_i \ge 0$ and 
\begin{equation}
\label{eq:firststep}
U^*\Phi(X)U =\sum_{i =1}^m W_i(U^*\Psi_i(X)U)W_i = \sum_{i =1}^m W_i\Psi_i(X)W_i.
\end{equation}
Since $\sum_{i=1}^m W_i^2 = I_H$, it follows that 
 $\sum_{i=1}^m Y_i^2 = I$. Since $\|Y_1\|<1$, 
$I-Y_1^2$ invertible. 
Also since 
 \[
\sum _{i\geq 2}\begin{pmatrix}
0 & 0\\ 
0 & Y_i^2
\end{pmatrix} = \sum_{i\geq 2} W_i^2=I_H - W_1^2= \begin{pmatrix}
0 & 0 \\ 
0 & I -Y_1^2
\end{pmatrix},
\]
it follows that $\sum_{i\geq 2}Y_i^2$ is invertible. 
Adapting Technique-C from \cite{FM93} to our setting and applying it here 
 allows us to write $\sum_{i\geq 2} W_i\Psi_i(X)W_i$ as a single term 
$T_0^*\Gamma_0(X)T_0$ for some $\Gamma_0\in \cE$. Note that $T_0$ need not be positive. 
By using the polar decomposition $T_0= U_0W_0$ where $U_0$ is unitary and
 $W_0$ is a positive operator, observe that 
$
T_0^*\Gamma_0(X)T_0=W_0\Psi_0(X)W_0, 
$
for some $\Psi_0 \in \cE$.
% Since $U_0$ is unitary, it follows that $U_0^*\psi_0(X)U_0 \in \cE$.
% and absorbing the unitary part of $T_0$ into $\psi_0$, 
Thus 
\begin{equation}
\label{eq:secstep}
\sum_{i\geq 2} W_i\Psi_i(X)W_i= W_0\Psi_0(X)W_0
\end{equation} 
for some positive operator $W_0$ and $\Psi_0\in \cE$.
%\\
%& = \begin{pmatrix}
%I &  0 \\ 
%0 &   Y_1
%\end{pmatrix} \psi_1(X) \begin{pmatrix}
%I & 0 \\ 
%0 &  Y_1
%\end{pmatrix}+ \sum_{i\geq 2} \begin{pmatrix}
%0 & 0\\ 
%0 & Y_i
%\end{pmatrix} \psi_i(X)\begin{pmatrix}
%0 & 0 \\ 
%0 & Y_i
%\end{pmatrix}.

%Let 
%$W_1= \begin{pmatrix}
%I & 0 \\ 
%0 & Y_1
%\end{pmatrix}$ and 
%$W_i=\begin{pmatrix}
%0 & 0\\ 
%0 & Y_i
%\end{pmatrix},$ for all $i\geq 2$. 
%Thus we have $U^*\Phi(X)U=\sum W_i\psi_i(X)W_i$.  
Indeed $W_0^2+W_1^2= I_H$ and $W_0= \begin{pmatrix}
0 & 0\\ 
0 & Y_0
\end{pmatrix}
$ for some positive operator $Y_0$. 
%Note that $W_1, W_0$ admit a further decomposition namely, can be further decomposed as 
Recall the positive matrix $Y_1$ and rewrite it as 
 $Y_1 = \begin{pmatrix}
Z_1 & 0\\ 
0 & 0
\end{pmatrix}$ where $Z_1$ is a positive invertible matrix with 
$\|Z_1\|<1.$ 
Since $Y_0^2 + Y_1^2 = I$, 
 $Y_0= \begin{pmatrix}
 Z_0 & 0 \\ 
0 & I
\end{pmatrix}$ 
for some matrix $Z_0$.
% where $Z_0$ is such that $Z_0^2 + Z_1^2 = I$.
Thus 
\begin{equation*}
W_1=\begin{pmatrix}
I & 0 & 0\\ 
0 & Z_1 & 0\\ 
0&0 & 0
\end{pmatrix} \quad \text{and} \quad W_0=\begin{pmatrix}
0 & 0 & 0\\  
0& Z_0 & 0\\ 
0&0 & I
\end{pmatrix}, 
\end{equation*}
where $Z_1, Z_0$ are positive matrices. Since $W_0^2 + W_1^2 = I_H$, 
it must be the case that $Z_0^2 + Z_1^2 = I$, which in turn implies that 
$Z_0$ is invertible, since $\|Z_1\|<1$.
Recall that $Z_1$ is also invertible. It follows from equations \eqref{eq:firststep} and 
\eqref{eq:secstep} and the fact that $\Psi_1 \in \cK$ that 
\begin{align}
\notag
&U^*\Phi(X)U =  W_0 \Psi_0(X) W_0 + W_1 \Psi_1(X) W_1  \\
%\begin{pmatrix}
%1 & \\ 
%& Z_1 & \\ 
%& & 0
%\end{pmatrix}\psi_1(X)\begin{pmatrix}
%1 & \\ 
%& Z_1 & \\ 
%& & 0
%\end{pmatrix}+ \begin{pmatrix}
%0 & \\  
%& Z_0 & \\ 
%& & 1
%\end{pmatrix} \psi_0(X) \begin{pmatrix}
%0 & \\  
%& Z_0 & \\ 
%& & 1
%\end{pmatrix}\\
\label{eq:matrix-one}
&= \begin{pmatrix}
I & 0 & 0  \\ 
0 & Z_0 & 0\\ 
0 & 0 & I\\ 
\end{pmatrix} \Theta(X) \begin{pmatrix}
I & 0 & 0  \\ 
0 & Z_0 & 0\\ 
0 & 0 & I\\ 
\end{pmatrix} + 
\begin{pmatrix}
0 & 0 & 0\\ 
0 & Z_1 & 0 \\ 
0 & 0 & 0
\end{pmatrix} \Psi_1(X) \begin{pmatrix}
0 & 0 & 0 \\ 
0 & Z_1 & 0\\ 
0 & 0 & 0
\end{pmatrix},
%&+ \begin{pmatrix}
%0 & \\ 
%& Z_0 & \\ 
%& & 1
%$\end{pmatrix} \psi_0(X)\begin{pmatrix}
%0 & \\ 
%& Z_0 & \\ 
%& & 1
%\end{pmatrix}
\end{align}
where 
\begin{align}
\label{eq:matrix-two}
\Theta(X) &= \begin{pmatrix}
I & 0 & 0 \\ 
0 & 0 & 0 \\ 
0 & 0 & 0
\end{pmatrix} \Psi_1(X) \begin{pmatrix}
I & 0 & 0\\ 
0 & 0 & 0 \\ 
0 & 0 & 0
\end{pmatrix}+ \begin{pmatrix}
0 & 0 & 0 \\
0 & I & 0\\
0 & 0 & I
\end{pmatrix} \Psi_0(X) \begin{pmatrix}
0 & 0 & 0\\
0 & I & 0 \\
0 & 0 & I
\end{pmatrix} \\
\notag
& \in \mathcal E.
\end{align}
By the invertibility of $Z_0$, it is immediate that 
$Y_0= \begin{pmatrix}
 Z_0 & 0 \\ 
0 & I
\end{pmatrix}$ is also invertible. Applying Technique-B from \cite{FM93} allows for writing 
$U^*\Phi U$ as a proper $C^*$-convex combination of $\Theta$  and some other 
$\Gamma \in \mathcal E$. By hypothesis, $\Phi$ is a $C^*$-extreme UEB map. Hence   
so is $U^*\Phi U$. It follows that $\Phi$ is unitarily equivalent to $\Theta$. Let $Q_1$ and 
$Q_1^\perp$ denote the projections $\begin{pmatrix}
I & 0 & 0 \\ 
0 & 0 & 0 \\ 
0 & 0 & 0
\end{pmatrix}$ and $\begin{pmatrix}
0 & 0 & 0 \\
0 & I & 0\\
0 & 0 & I
\end{pmatrix}$ respectively.
Indeed, with respect to the decomposition $H = \range(Q_1) \oplus \range(Q_1^\perp)$,  
\begin{equation}
\label{eq:dirsum}
\Theta(X) = Q_1 \Psi_1(X)|_{\range(Q_1)} \oplus Q_1^\perp \Psi_0(X)|_{\range(Q_1^\perp)},
\end{equation}
for all $X \in \cS$.
Since  $Q_1^\perp \Psi_0(X)|_{\range(Q_1^\perp)} \in$ UEB($\cS, B(\range(Q_1^\perp)))$, one 
can write 
%one can write $Q_1^\perp \psi_0(X)|_{\range(Q_1^\perp)}$ 
%as a $C^*$-convex combination of UEB maps of the form 
%$\Gamma(X) = \gamma(X)_I{\range(Q_1^\perp)}$, 
%where $\gamma$ is a linear extremal state defined on $\cS$. i.e., 
\[
Q_1^\perp \Psi_0(X)|_{\range(Q_1^\perp)} := \sum_{i=1}^k S_i^* \Gamma_i(X)S_i,
\]
where $\Gamma_i \in$ UEB($\cS, B(\range(Q_1^\perp)))$ are given by 
$\Gamma_i(X) = \gamma_i(X) I_{\range(Q_1^\perp)}$, $\gamma_i$ are 
linear extremal states defined on $\cS$, $S_i \in B(\range(Q_1^\perp))$ satisfy
 $\sum_{i=1}^k S_i^*S_i = I_{\range(Q_1^\perp)}.$
Since $\Theta \in$ UEB($\cS, B(H))$ is a $C^*$-extreme map, it follows from 
 an argument similar to the one on page 770 in \cite{FM93}, that each direct summand 
of $\Theta$ in equation \eqref{eq:dirsum} is also a $C^*$-extreme UEB map. In particular, 
the UEB map $Q_1^\perp \Psi_0(X)|_{\range(Q_1^\perp)}$ is a $C^*$-extreme point of  
UEB($\cS, B(\range(Q_1^\perp)))$. As before, there is no loss of generality in
 assuming that $\|S_1\|=1$. We can now repeat the arguments from before to the 
current set up and conclude (by taking appropriate direct sums with the 
zero operator) that there exists a projection $Q_2$ in $B(H)$ such that 
\[
\Theta(X) = Q_1 \Psi_1(X) Q_1 + Q_2 \Delta_1(X) Q_2 + \text{an EB map}, 
\]
where $Q_1$ and $Q_2$ are mutually orthogonal projections in $B(H)$,
 $\Delta_1(X) = \gamma_1(X)I_H$ and 
$\gamma_1$ is a linear extremal state on $\cS$.
 This process has to end after a finite number of steps, 
due to our finite dimensionality assumptions. This along with the fact that $\Theta$ is 
unitarily equivalent to $\Phi$ implies that $\Phi(X) = \sum_{i=1}^\ell \phi_i(X)P_i$, for 
all $X \in \cS$, where the $\phi_i$'s are linear extremal states on $\cS$ and the $P_i$'s
 are mutually orthogonal projections in $B(H)$ satisfying $\sum_{i=1}^\ell P_i = I_H$. 
%Without loss of generality one can assume that the $\phi_i$'s are distinct. 
An application of 
Theorem \ref{thm:max-iff-stdform} implies that $\Phi$ is a maximal UEB map and 
the proof is complete. 

To prove the implication (ii) $\implies$ (iii), observe that since $\Phi$ is a $C^*$-extreme 
UEB map, by the (proved) equivalence of statements (i) and (ii) and by an application 
of Lemma \ref{lem:maximal-imp-commrang}, $\Phi$ has commutative 
range. That $\Phi$ is a linear extreme UEB map follows from an easy and direct  
 adaptation of Theorem 2.2.2 in \cite{Z} to UEB maps.

Finally, to prove the implication (iii) $\implies $ (i), observe that since $\Phi$ is a 
UEB map with commutative range,  Lemma \ref{lem:comm-range} implies that 
\[
\Phi(X)=\sum_{i=1}^k\phi_i(X)P_i,
\]
for some $k \in \N$, where the $\phi_i$'s are distinct
 states on $\cS$ and the $P_i$'s are mutually orthogonal projections in 
$B(H)$ satisfying $\sum_{i=1}^k P_i=I_H$. 
By appealing to Theorem \ref{thm:max-iff-stdform}, it 
suffices to show that $\phi_i$ is a linear extremal state for each $i$.
 Fix $i\in \{1, \dots, k\}$. Let $\sigma$ and $\tau$ be states defined on $\cS$ such that 
\[
\phi_i=t \sigma +(1-t)\tau,
\]
for some $t\in (0,1)$. It is enough to show that $\sigma=\tau$. 
Observe that \begin{align*}
\Phi(X)= t\left( \sigma(X)P_i+ \sum_{j\neq i}\phi_j(X)P_j\right)+(1-t)\left(\tau(X)P_i+ \sum_{j\neq i}\phi_j(X)P_j\right),
\end{align*}
for all $X\in \cS$. Define the linear maps $\Psi, \Gamma:\cS \rightarrow B(H)$ by 
\[
\Psi(X) = \sigma(X)P_i+ \sum_{j\neq i}\phi_j(X)P_j \text{ and }
\Gamma(X) = \tau(X)P_i+ \sum_{j\neq i}\phi_j(X)P_j.
\] Observe that 
$\Psi, \Gamma \in $ UEB$(\cS, B(H))$ and 
$\Phi(X)= t \Psi(X)+(1-t)\Gamma(X)$, for all $X\in \cS$. 
 From the assumption that $\Phi$ is linear extreme,
it follows that $\Psi = \Gamma$. This in turn implies that $\sigma= \tau$ 
and the proof is complete. \qedhere
\end{proof}

\begin{remark}
The following remarks concern the above proof of Theorem \ref{thm:max-iff-C*ext}.
\begin{itemize}
\item[(i)] The proof of the implication $(i) \Rightarrow (ii)$ in the above theorem works equally well
 for UCP maps or even just unital positive maps. 
\item[(ii)] In the proof of the implication $(ii) \implies (i)$, it is assumed that $T_1$ is not invertible and 
hence the appearance of the zero block in the definition of $Y_1$. If the operator $T_1$ is invertible, 
even then the proof works just fine by letting $Y_1 = Z_1$, $Y_0 = Z_0$ and by deleting the last row 
and column in the coefficient matrices occuring in equations 
\eqref{eq:matrix-one} and \eqref{eq:matrix-two}.
\end{itemize}
\end{remark}

%As an application of the above theorem, we show that every $C^*$-extreme UEB map on $\cS$ extends to a $C^*$-extreme UEB map.
%% point of UEB$(B(E), B(H))$.
%% map can extend to $C^*$-extreme UEB map.
%
%\begin{corollary}
%Let $\cS \subset B(E)$ be an operator system. Then a $C^*$-extreme UEB map $\Phi : \cS \to B(H)$ can be extended to a $C^*$-extreme UEB map $\Psi :B(E) \to B(H)$.
%\end{corollary}
%The following are proofs of the applications of our main results. 

\begin{proof}[Proof of Corollary \ref{cor:extn}]
Since $\Phi$ is a $C^*$-extreme point in UEB$(\cS, B(H))$, it follows from 
Theorems \ref{thm:max-iff-C*ext} and \ref{thm:max-iff-stdform} that
$
\Phi(X)= \sum_{i=1}^\ell \phi_i(X)P_i,
$
where the $\phi_i$'s are linear extremal states on $\cS$ and the 
$P_i$'s are orthogonal projections in $B(H)$ such that $\sum_{i=1}^\ell P_i=I_H$.
 By \cite[Proposition 1.2.4]{Z} there exist linear extremal states  $\psi_i:B(E) \to \C$
 such that $\psi_i\big|_\cS=\phi_i$, for each $i$. Define $\Psi:B(E) \to B(H)$ by 
$
\Psi(X)= \sum_{i=1}^\ell \psi_i(X)P_i.
$
Observe that $\Psi\big|_\cS=\Phi$. It follows from Theorems \ref{thm:max-iff-stdform} 
and \ref{thm:max-iff-C*ext} that $\Psi$ is a $C^*$-extreme UEB extension of $\Phi$.
\end{proof}

The following example shows 
 that the converse of Corollary \ref{cor:extn} is not true in general,
i.e., the restriction of a $C^*$-extreme UEB map on an 
operator system need not always be $C^*$-extreme.  
Recall that for states defined on operator systems, $C^*$-extremality 
coincides with linear extremality. Let $\cS \subset M_2(\C)$ denote 
the operator system  
\[
\cS:=\left\lbrace \begin{pmatrix}
a & b \\
c & a
\end{pmatrix} \in M_2(\C) : a,b,c \in \C\right\rbrace.
\]
Define $\phi: \cS \to \C$ by $\phi(X)= \trace(XE)$ for all $X\in \cS$, where $E= \begin{pmatrix}
1/2 & 0 \\
0 & 1/2
\end{pmatrix}$. Then $\phi$ is a state on $\cS$. In particular, $\phi$ is a UEB map. 
Define the (distinct) states $\phi_1, \phi_2 : \cS \to \C$ by 
$\phi_i(X)= \trace(XF_i)$,  for all $X\in \cS$, i = 1, 2, 
where $F_1= \begin{pmatrix}
1/2 & 1/2\\
1/2 & 1/2
\end{pmatrix}$ and $F_2= \begin{pmatrix}
1/2 & -1/2\\
-1/2 & 1/2
\end{pmatrix}$.
It follows that 
 \begin{align*}
\phi(X) &= \trace\left(X \begin{pmatrix}
1/2 & 0 \\
0 & 1/2
\end{pmatrix}\right)
= \frac{1}{2} \trace(X(F_1+F_2)) \\
& = \frac{1}{2} \left(\phi_1(X)+\phi_2(X)\right).
\end{align*}
This in turn implies that $\phi$ is not a linear extremal state and hence 
 not a $C^*$-extreme state on $\cS$. Let $\psi:M_2 \to \C$  be the state defined by 
$\psi(Y)= \trace(YE_{11})$,  where 
$E_{11}= 
\begin{pmatrix}
1 & 0\\
0 & 0
\end{pmatrix}$. Since $E_{11}$ is a projection of rank one, it is well-known that 
$\psi$ has to be a linear extremal state (See \cite{St} and \cite{Wa})
or equivalently a $C^*$-extremal state on $M_2$. 
Finally, observe that $\psi|_{\cS}=\phi$. \\

We end this section with a proof of Corollary \ref{cor:kre-mil-for-ueb}, which is 
a Krein-Milman type theorem for UEB maps. 
\begin{proof}[Proof of Corollary \ref{cor:kre-mil-for-ueb}]
 Recall the notations $\cE$ and $\mathcal K$ from Lemma \ref{lem:ueb}. It follows 
from Lemma \ref{lem:ueb} that 
$\cE = $ UEB($\cS, B(H))$. By Theorems \ref{thm:max-iff-stdform} and 
\ref{thm:max-iff-C*ext}, it follows that each element of $\cK$ is a $C^*$-extreme 
UEB map in UEB($\cS, B(H))$. This completes the proof. 
\end{proof}

\section{$C^*$-extreme UEB maps between Matrix algebras} 
\label{S:matalg}
%
%The entanglement breaking rank (or EB-rank) of a given EB map $\Phi:B(E) \to M_n$ is defined as the minimum
% number of rank one Choi-Krauss operators required to represent the map $\Phi$ in
% Choi-Krauss representation and is denoted by EB-rank($\Phi$), please see \cite{PPPR}.
% Similarly, if $\Phi: B(E) \to M_n$ is a CP map then the minimum number of (linearly
% independent) Choi-Krauss operators required to represent $\Phi$ is called
% Choi-rank of $\Phi$. 
This section contains an improved version of \cite[Theorem 5.3]{BDMS}, 
which includes various characterizations of a $C^*$-extreme UEB map 
between matrix algebras (See Theorem \ref{thm:improv}).
Recall the convention~\ref{assump:main} and the definitions of the Choi-rank, 
EB rank, Schmidt rank and Schmidt number 
from Section \ref{sec:intro}. For a positive matrix 
$X = \sum_{i=1}^\ell A_i \otimes B_i \in 
B(E \otimes H)$, \df{the partial trace of $X$ with respect to 
the first coordinate}, is denoted by $\trace_1(X)$ and 
is defined as 
\[
\trace_1(X) = \displaystyle \sum_{i=1}^\ell \trace(A_i)B_i \in B(H).\\
\]
The following lemma is a minor variant of \cite[Lemma 8]{HSR}.
\begin{lemma}
\label{lem:spectral-decomp-choi}
Let $\Psi:M_d \rightarrow M_n$ be an EB map and $C_{\Psi} = 
\sum_{j=1}^m \xi_j\xi_j^*$ with $SR(\xi_j)=1$ for each $j \in \{1,2,\dots,n\}$. If the Choi-rank of $\Psi = n$ and 
$\trace_1(C_{\Psi}) = \Psi(I_d) \in M_n$ is invertible, then $m \ge n$ and
 there exists $\gamma_1,\dots,\gamma_n\in \C^d \otimes \C^n$ with $SR(\gamma_k) = 1$ such that 
$C_{\Psi} = \sum_{k=1}^n \gamma_k \gamma_k^*$. 
\end{lemma}

\begin{proof}
Since the Choi-rank of $\Psi$ equals the rank of the Choi matrix $C_\Psi$ (See \cite{Wa} and \cite{St}), it is clear that $m \ge n$.
 Suppose that $m > n$ and $C_{\Psi} = \sum_{j=1}^m \xi_j \xi_j^*$ 
where $SR(\xi_j) = 1$ for all $1 \le j \le m$.  For each $1 \le j \le m$, let $\xi_j = 
x_j \otimes y_j$ with $\|x_j\|=1$. By hypothesis, $\Psi(I_d) = \trace_1(C_{\Psi}) = 
 \sum_{j=1}^m \trace(x_jx_j^*) y_jy_j^* =  \sum_{j=1}^m y_jy_j^*$ is invertible. This implies that 
Span$\{y_1,\dots,y_m\} = \C^n$. Without loss of generality, assume that 
$\{y_1,\dots,y_n\}$ is a basis for $\C^n = \range(\trace_1(C_{\Psi}))$. 
%This in turn implies that 
 Note that 
 $B=\{\xi_j\,:\, 1 \le j \le n\}$ is linearly independent.  By the Douglas range 
inclusion property (\cite[Theorem 1]{D}) and due to the hypothesis that the rank of $C_{\Psi}=n$,
 it follows that $\{\xi_j\,:\, 1 \le j \le m\} \subset \range(C_{\Psi})$ and that 
 $B$ is in fact a basis for $\range(C_{\Psi})$. 
Consider the sum $\sum_{j=1}^{n+1 }\xi_j \xi_j^*$. 
There exists $\alpha_j$ such that $\xi_{n+1} = \sum_{j=1}^n \alpha_j \xi_j$. 
It follows that 
$x_j = \lambda_j x_{n+1}$, for some scalars $\lambda_j$, 
whenever $\alpha_j \neq 0$, $1 \le j \le n$.  Thus
$|\lambda_j| = 1$ (since the $x_j$'s are unit vectors) 
and 
\[
x_{n+1}\otimes y_{n+1}= \xi_{n+1} = \sum_{j:\alpha_j\ne 0} \alpha_j \xi_j 
  = \sum_{j:\alpha_j\ne 0} x_{n+1}\otimes (\lambda_j \alpha_j) y_j.
\]
Consequently,
$y_{n+1} = \displaystyle \sum_{j: \alpha_j \neq 0} (\lambda_j\alpha_j) y_j,$ and  
\begin{align}
\notag
\sum_{j=1}^{n+1 }\xi_j \xi_j^* &=  x_{n+1}x_{n+1}^* \otimes \left( \left( \sum_{j:\alpha_j \neq 0} 
%|\lambda_j|^2  
y_jy_j^* \right)  + y_{n+1}y_{n+1}^* \right) \\
\label{eq:part-sum}
& + \sum_{j:\alpha_j = 0} x_jx_j^* \otimes y_jy_j^*.
\end{align}
% Observe that $\trace_1 \left(\sum_{j=1}^{n+1} \xi_j \xi_j^* \right)  \le \trace_1(C_{\Psi})$, 
% %Observe that the partial trace 
% %\[ 
% %\trace_1(C_{\Psi}) = \trace_1 \left(\sum_{j=1}^{n+1} \xi_j \xi_j^* \right) + \sum_{j=n+2}^m y_jy_j^*
% %\]
% where 
% %$\trace_1 \left(\sum_{j=1}^{n+1} \xi_j \xi_j^* \right) \le \trace_1(C_{\Psi})$
% \begin{equation}
% \label{eq:part-trace}
% \trace_1 \left(\sum_{j=1}^{n+1} \xi_j \xi_j^* \right)  = \sum_{j:\alpha_j = 0} y_jy_j^* + \sum_{j:\alpha_j \neq 0} y_jy_j^* + y_{n+1}y_{n+1}^*.
% \end{equation}
% %It is immediate from equations \eqref{eq:part-sum} and \eqref{eq:part-trace} that $| \lambda_j |=1$ 
% %for each $j$ such that $\alpha_j \neq 0$. 
% It follows from equation \eqref{eq:part-trace} and 
% the Douglas property (\cite[Theorem 1]{D}) that $\{y_1, \dots y_n\} \subset \range \left(\trace_1 \left(  \sum_{j=1}^{n+1} \xi_j \xi_j^* \right) \right)$.
% Thus the ranges of  $\trace_1 \left(\sum_{j=1}^{n+1} \xi_j \xi_j^* \right )$ and $\trace_1(C_{\Psi})$ coincide. In particular,  
% \begin{equation}
% \label{eq:rank}
% \rankk \left (\trace_1 \left (\sum_{j=1}^{n+1} \xi_j \xi_j^* \right) \right ) = \rankk \left(\trace_1(C_{\Psi}) \right) = n.
% \end{equation}

Let $r$ denote the cardinality of the set $\{j:\alpha_j=0\}$. 
%{ Observe that 
%$\rankk \left (\trace_1 \left (\sum_{j=1}^{n+1} \xi_j \xi_j^* \right) \right ) 
%= \rankk\left(\trace_1(C_{\Psi}) \right) = n.$}
%{\crd
Since $\{y_j:\alpha_j \neq 0\}$ is linearly independent and   
 $y_{n+1} \in \operatorname{span}\{y_j : \alpha_j \neq 0\}$, 
\[
\sum_{j:\alpha_j \neq 0}  y_jy_j^* + y_{n+1}y_{n+1}^* = 
\sum_{k=1}^{n-r}w_kw_k^*,
\]
for some non-zero vectors $w_k \in \operatorname{span}\{y_j: \alpha_j\ne 0\}\subseteq \C^n.$ 
Thus,
\[
\sum_{j=1}^{n+1} \xi_j \xi_j^* = \sum_{k=1}^{n-r} \eta_{k}\eta_{k}^* + \sum_{j:\alpha_j=0} \xi_j\xi_j^*,
\]
where $\eta_k = x_{n+1} \otimes w_k$.
Note that the right hand side is a sum consisting of exactly $n$ summands that are vectors in
 $\C^d \otimes \C^n$ 
with Schmidt rank equal to one. Repeat the above process, by adding the $(n+2)^{th}$ term of 
$C_{\Psi}$ to $\sum_{j=1}^{n+1} \xi_j \xi_j^* $. Continuing this way, 
after a finite number of iterations, one can conclude that $C_{\Psi}$ is of the desired form, i.e.,  
$C_\Psi =\sum_{k=1}^n \gamma_k\gamma_k^*$ with $\gamma_k \in \C^d \otimes \C^n$ and 
$SR(\gamma_k) = 1$. This completes the proof. 
\end{proof}
%The following theorem is an improved version of \cite[Theorem 5.3]{BDMS}. 
\begin{theorem}
\label{thm:improv}
Let $\Phi:M_d \to M_n$ be a UEB map. The following statements are equivalent.
\begin{itemize}
\item[(i)]  $\Phi$ is a $C^*$-extreme UEB map.
\item[(ii)] $\Phi$ is a linear extreme UEB map with commutative range.
\item[(iii)] $\Phi$ is a maximal UEB map.
\item[(iv)] The map $\Phi$ has the form
%\begin{equation*}
%\label{eq:stdform-of-phi}
$\Phi(X)= \sum_{i=1}^{\ell} \phi_i(X)P_i,$
%\end{equation*} 
where the $\phi_i$'s are distinct linear extremal states defined on $M_d$ and the $P_i$'s 
are mutually orthogonal projections in $M_n$ satisfying $\sum_{i=1}^{\ell} P_i=I_n$.
\item[(v)] The Choi-rank of $\Phi$ is $n$.
\item[(vi)]
% Every (spectral) decomposition of $C_{\Phi}$ as a sum of rank one positive matrices 
%of the form $\xi\xi^*$ with $\xi \in \C^d \otimes \C^n$ satisfying $SR(\xi) = 1$, can be 
%transformed into a sum with exactly $n$ such summands. i.e., 
If $C_{\Phi} = \sum_{j=1}^m \xi_j\xi_j^*$ with $SR(\xi_j)=1$ for
each $j \in \{1,2,\dots,m\}$, then $m \ge n$ and there exists $\gamma_1,\dots,\gamma_n \in \C^d \otimes \C^n$ 
with $SR(\gamma_j) = 1$ such that $C_{\Phi} = \sum_{j=1}^n \gamma_j \gamma_j^*$.
\item[(vii)] 
%Every Choi-Kraus representation of $\Phi$ with rank-one Choi-Kraus coefficients, 
%can be transformed into another such representation with exactly $n$ summands, 
%i.e.,
 If $\Phi(X) = \sum_{i=1}^k V_i^*XV_i$ and 
rank of $V_i$ is one for each $i \in \{1,2,\dots,k\}$, then $k \ge n$ and 
there exists matrices
$W_1,\dots,W_n$ of rank one such that 
$\Phi(X) = \sum_{i=1}^n W_i^*XW_i$.
\item[(viii)] The EB-rank of $\Phi$ is $n$.
%\item[(vii)]  $\Phi$ is a $C^*$-extreme UEB map.
\end{itemize}
\end{theorem}

\begin{proof}
The equivalence of $(i), (ii)$ and $(iii)$ follows by letting $\cS = B(E)$ in
 Theorem \ref{thm:max-iff-C*ext}

The equivalence $(iii) \iff (iv)$ follows by letting  
$\cS=B(E)$ in Theorem \ref{thm:max-iff-stdform}. 
%------

To prove $(iv) \implies (v), $ rewrite $\Phi(X) = \sum_{k=1}^n \psi_k(X) w_kw_k^*$,
 where $\psi_k \in \{\phi_i\,:\, 1 \le i \le \ell\}$ and $\{w_k\,:\, 1 \le k \le n\}$ is an orthonormal 
set of vectors satisfying $\sum_{k=1}^n w_kw_k^* = I_n$. For each $k \in \{1, 2, \dots, n\}$, 
observe that $\psi_k(X) = \trace(Xu_ku_k^*) = \langle X u_k, u_k \rangle$ for some unit vector $u_k$.
 Recall the Choi matrix $C_{\Phi} \in M_d \otimes M_n$. It follows that 
\begin{align*}
C_{\Phi} &= \sum_{i,j=1}^d e_ie_j^* \otimes \Phi(e_ie_j^*)
= \sum_{k=1}^n \sum_{i,j=1}^d e_ie_j^* \otimes  (u_k^*e_ie_j^*u_k) w_kw_k^*\\
              %& = \sum_{k=1}^n \left( \sum_{i=1}^d e_i \otimes (u_k^*e_i)w_k \right)
% \left( \sum_{i=1}^d e_j^* \otimes (e_j^*u_k)w_k^* \right).\\
              & =   \sum_{k=1}^n \left( \sum_{i=1}^d (u_k^*e_i) e_i \otimes w_k \right) 
 \left( \sum_{i=1}^d (e_i^*u_k) e_i^*  \otimes w_k^* \right) = \sum_{k=1}^n z_kz_k^*,
\end{align*}
where $z_k =  \sum_{i=1}^d (u_k^*e_i) e_i \otimes w_k $ . Since 
$C_{\Phi}$ is a sum of rank-one 
operators,  the rank of $C_{\Phi}$, i.e., the Choi-rank of $\Phi$ is at most $n$. 
Also since $C_{\Phi}$ is a sum of the positive rank-one operators $z_kz_k^*$, the 
Douglas range inclusion property (\cite[Theorem 1]{D})  implies that 
$\{z_k\,:\, 1 \le k \le n\}$ is contained in $\range(C_{\Phi})$. By the orthonormality of 
$\{z_k\,:\, 1 \le k \le n\} \subset \range(C_{\Phi})$ it follows that Choi-rank$(\phi) \ge n$.
 Thus in fact, the Choi-rank of $\Phi$ equals $n$.

Since $\Phi$ is unital, the implication $(v) \implies (vi)$ follows from the 
 observation that $\trace_1(C_{\Phi})=I_n$ and a direct application 
 of Lemma \ref{lem:spectral-decomp-choi}.

The implication $(vi) \implies (vii)$, is a direct consequence of the well-known fact
 (See \cite[Proposition 4.1.6]{St} and \cite[Theorem 3.1.1]{Bh} for instance),
 that every spectral decomposition of $C_{\Phi}$ as a sum of $\ell$ rank one positive
 matrices $\xi\xi^*$ with 
$SR(\xi) = 1$, yields a Choi-Kraus decomposition of $\Phi$ with exactly $\ell$ Choi-Kraus coefficients, 
each having rank one (and vice-versa).

The implication $(vii) \implies (viii)$ is immediate from the definition of the EB-rank of $\Phi$. 

The below proof of the implication $(viii) \implies (iv)$ is essentially the same as that given in 
\cite{BDMS}. We include it here for the sake of completeness. Let $\Phi(X) = \sum_{j=1}^n (v_ju_j^*) X (u_jv_j^*) $, 
where $\{u_1,\dots,u_n\} \subset \C^d$ are unit vectors and $\{v_1,\dots,v_n\} \subset \C^n$. 
It suffices to show that $\{v_1,\dots,v_n\}$ forms an orthonormal basis for $\C^n$ or 
equivalently, the matrix $W := \left( v_1 \, \vert \,  v_2 \, \vert  \cdots \vert \, v_n \right) \in M_n$ is unitary. 
This easily follows from the fact that $I_n = \Phi(I_d) = \sum_{j=1}^n v_jv_j^* = WW^*$. Let 
$\{u_{j_k}:1 \le k \le \ell\}$ be the distinct unit vectors among $u_1,\dots,u_n$. 
For $1 \le k \le \ell$, define the states $\phi_k$ on $M_d$ by $\phi_k(X) = \trace(Xu_{j_k})$. 
Rewrite $\Phi$ such that
\[
 \Phi(X) = \sum_{j=1}^{n} (v_ju_j^*) X (u_jv_j^*) = \sum_{k=1}^{\ell} \phi_k(X)P_k.
\]
It is easily seen that the $P_k$'s are mutually orthogonal projections satisfying 
$\sum_{k=1}^{\ell} P_k = I_n$. This completes the proof.
\end{proof}

\section{An example}
\label{s:eg}
This section contains an example of a $C^*$-extreme UEB map on an 
operator system. We begin by recalling some basic facts about states.
  Let $\cS \subset M_d$ be an operator system and 
$\psi:\cS \rightarrow \C$ be a state. By the Riesz representation 
theorem, there exists a unique $L \in \cS$ such that 
$\psi(X) = \trace(XL^*)$ for all $X \in \cS$.
Observe that $\trace(L) = 1$. Moreover, for each $X \in \cS$,  
\[
\trace(XL^*) = \psi(X) = \overline{\psi(X^*)} = \overline{\trace(X^*L^*)} = \trace(XL).
\]
By the uniqueness of the Riesz representative of $\psi$, $L^* = L$. We 
record these basic facts below. 

\begin{remark}
\label{rem:on-states}
Let $\cS \subset M_d$ be an operator system. 
If $\psi:\cS \rightarrow \C$ is a state and $L$ is the Riesz 
representative of $\psi$, then $\trace(L) = 1$ and $L = L^*$.
\end{remark}
\noindent
{\bf Example:} Let $\cS \subset M_3$ denote the operator system  
\[
\cS= \left\{\begin{pmatrix}
x & 0 & y\\
0 & x & z\\
u & v & w
\end{pmatrix} \,:\, x,y,z,u,v,w \in \C  \right\}.
\]
Define the linear map $\Phi: \cS \to M_n$ by 
\[
\Phi(X)= \phi(X)I_n,
\]
where $\phi: \cS \to \C$ is the linear functional 
\[
\phi(X)= \trace (XF),
\] 
and $F= \begin{pmatrix}
1/2 & 0 & 0\\
0 & 1/2 & 0\\
0 & 0 & 0
\end{pmatrix}$. 
Note that $\phi(X)$ is nothing but the $(1,1)$ entry of the matrix $X$. 
It is easily seen that $\phi$ is, in fact, a state on $\cS$. 
Since $\Phi$ is a unital positive map with commutative range, by Lemma \ref{lem:pos-to-eb},  
it follows that $\Phi$ is a UEB map on the operator system $\cS$.

Our goal is to show that $\Phi$ is, in fact, a $C^*$-extreme UEB map on $\cS$. 
By appealing to Theorems \ref{thm:max-iff-stdform} and \ref{thm:max-iff-C*ext}, 
it is enough to show that $\phi$ is a linear extremal state on $\cS$. To this end, let 
\[
\phi= \frac{1}{2} \left(\phi_1+\phi_2\right), 
\]
where $\phi_i: \cS \to \C$ are states on $\cS$, for $i=1,2$. It suffices to show that 
$\phi_1 = \phi_2$. By the Riesz Representation theorem, there exists unique $Y, Z \in \cS$ such that 
$\phi_1(X)= \trace(XY^*)$ and $\phi_2(X)= \trace(XZ^*)$ for all $X\in \cS$. 
Since $\phi_1$ and $\phi_2$ are states on $\cS$, by Remark \ref{rem:on-states},
it follows that $Y=Y^*$, $Z=Z^*$ and $\trace(Y) = \trace(Z) = 1$. By 
the uniqueness of the Riesz representative of the state $\phi$, it must be the case that 
$F= \frac12 (Y+Z)$.\\
Let $Y :=\begin{pmatrix}
a & 0 & b\\
0 & a & c\\
\bar{b} & \bar{c} & d
\end{pmatrix}$ and $Z := \begin{pmatrix}
r & 0 & s\\
0 & r & t\\
\bar{s} & \bar{t} & u
\end{pmatrix}$. 
Note that 
\begin{equation}
\label{eq:entries of Y and Z}
a+ r=1 \text{ and } b+s=c+t=d+u =0.
\end{equation}
Moreover, 
\begin{equation}
\label{eq:trace-and-entries}
2a+d=2r+u=1.
\end{equation}
Since $\phi_1$ and $\phi_2$ are states, it follows that $\phi_1(e_3e_3^*) = d \geq 0$ and 
$\phi_2(e_3e_3^*) = u \geq 0$, where $e_3^* = (0, 0, 1)$. Combining this with equations  
\eqref{eq:entries of Y and Z} and \eqref{eq:trace-and-entries} yields 
\begin{equation}
\label{eq:entries}
 u = d = 0 \text{ and } a = r = \frac12. 
\end{equation}
Fix $\lambda >1$. Suppose that $b,c \neq 0$.  Consider the matrices  
\[
V= \begin{pmatrix}
1 & 0 & \frac{-\lambda b}{|b|}\\
0 & 1 & 0\\
\frac{-\lambda \bar{{b}}}{|b|} & 0 & \lambda^4 
\end{pmatrix} 
\text{   and   } 
W= \begin{pmatrix}
1 & 0 & 0\\
0 & 1 & \frac{-\lambda c}{|c|}\\
0 & \frac{-\lambda \bar{{c}}}{|c|} & \lambda^4
\end{pmatrix} \in \cS.
\]
Observe that $V, W \in \cS^+$.
Since $\phi_1$ is a state on $\cS$, it follows that 
\begin{equation}
\label{eq:pos-of-phi-one-Z}
0 \leq \phi_1(V) = \trace(VY) = 1-2 \lambda |b| 
\end{equation} 
and 
\begin{equation}
\label{eq:pos-of-phi-one-W}
0 \leq \phi_1(W) = \trace(WY) = 1-2 \lambda |c|.
\end{equation}
Letting $\lambda \rightarrow \infty$ in equations \eqref{eq:pos-of-phi-one-Z} 
and \eqref{eq:pos-of-phi-one-W} yields contradictions. 
Thus $b = c = 0$. It follows from equation  
\eqref{eq:entries of Y and Z} that $s = t = 0.$ 
Thus $Y = Z = F$. Equivalently, $\phi_1 = \phi_2 = \phi$. \\

\noindent
\textbf{Acknowledgements:} The authors thank Professor Scott McCullough (UF) for several discussions and many helpful suggestions, Dr. Devendra Repana (IITM) for a clarification/suggestion on Lemma 6.1 and also for pointing us to some important references and Mr. Chinmay Ajay Tamhankar (IITM) for his comments/suggestions on Theorem 1.10. 
The authors also thank the anonymous referees for their insightful comments and suggestions.
%many useful comments which improved the overall presentation and quality of this article. 

\end{document}